\documentclass[reqno,  11pt]{amsart}

\usepackage[top=3cm, bottom=2.5cm, left=2.5cm, right=2.5cm]{geometry}
\usepackage{enumerate}
\usepackage{amssymb,amsmath}
\usepackage{hyperref}
\newtheorem{thm}{Theorem}

\newtheorem{lem}{Lemma}

\theoremstyle{definition}

\theoremstyle{remark}

\allowdisplaybreaks

\newcommand{\R}{\mathbb{R}}

\newcommand{\E}{\mathbb{E}}

\numberwithin{equation}{section} \numberwithin{lem}{section}
\numberwithin{thm}{section} \numberwithin{prop}{section}
\numberwithin{cor}{section} \numberwithin{rem}{section}

\title[kinetic flocking models with local alignment]{Error Estimation in the Mean-Field Limit of Kinetic Flocking Models with Local Alignments}
\author{Jinhuan Wang$^{\,1}$, Keyu Li$^{\,1}$, and Hui Huang$^{\,2}$}
\thanks{The work of J. Wang is partially supported by National Natural Science Foundation of China Grants No. 12171218, LiaoNing Revitalization Talents Program Grant No. XLYC2007022.}
\thanks{Corresponding author: Hui Huang}
\begin{document}
\maketitle
\begin{center}
{\footnotesize
1-School of Mathematics, Liaoning University, Shenyang, 110036, P. R. China \\
email:  wjh800415@163.com; lky\_lnu@163.com\\
 \smallskip
2-Department of Mathematics and Scientific Computing, Karl-Franzens-Universit\"{a}t Graz, 8010 Graz, Austria
\\
email: hui.huang@uni-graz.at
}
\end{center}
\maketitle
\date{}
\begin{abstract}
In this paper, we present an innovative particle system characterized by moderate interactions, designed to accurately approximate kinetic flocking models that incorporate singular interaction forces and local alignment mechanisms. We establish the existence of weak solutions to the corresponding flocking equations and provide an error estimate for the mean-field limit. This is achieved through the regularization of singular forces and a nonlocal approximation strategy for local alignments. We show that, by selecting the regularization and localization parameters logarithmically with respect to the number of particles, the particle system effectively approximates the mean-field equation.

\end{abstract}


\newtheorem{theorem}{Theorem}[section]
\newtheorem{definition}{Definition}[section]
\newtheorem{lemma}{Lemma}[section]
\newtheorem{proposition}{Proposition}[section]
\newtheorem{corollary}{Corollary}[section]
\newtheorem{remark}{Remark}[section]
\renewcommand{\theequation}{\thesection.\arabic{equation}}
\catcode`@=11 \@addtoreset{equation}{section} \catcode`@=12

\section{INTRODUCTION}
In this work, we explore the following kinetic Cucker-Smale (C-S) flocking model characterized by its inclusion of confinement, a nonlocal interaction, a local alignment force, and a diffusion component
	\begin{eqnarray}\label{pde1}
    &\partial_t f+v\cdot\nabla_x f- \nabla_v\cdot[\big(\gamma v+\lambda(\nabla_x V +\nabla_x W*\rho)\big)f]=\nabla_v\cdot[\beta(v-u)f+\sigma\nabla_vf],
    \end{eqnarray}
subject to the initial data
	$$f(x,v,0) = f_0, \quad (x,v)\in \R^d\times \R^d.$$
Here $d\ge 3$, $f = f (x,v,t)$ is the particles distribution function at $(x, v,t)\in \mathbb R^d\times \mathbb R^d\times\mathbb R_+$, $\rho=\rho(x,t)$ and $u=u(x,t)$ are local particles density and velocity respectively, which are given by 
$$ \rho = \int_{\R^d} f dv\quad \mbox{and} \quad u=\frac{\int_{\mathbb{R}^d} vf dv}{\rho}.$$
The function $V(x)=|x|^2/2$ is a smooth confinement potential and $W(x)$ is an interaction potential, which is chosen to be the Newtonian potential case as following form
\begin{align*}
W(x) = \frac{1}{(d-2)|B(0,1)|}\frac{1}{|x|^{d-2}},
\end{align*}
where $|B(0,1)|$ denotes the volume of unit ball $B(0,1)$ in $\mathbb{R}^d$, i.e. $|B(0,1)| = \pi^\frac{d}{2}/\Gamma(d/2+1)$, $\Gamma(\cdot)$ denotes the Gamma function.  In (\ref{pde1}), the first two terms account for the free transport of particles, while the third term consists of linear damping with strength $\gamma > 0$ and the confinement and interaction forces due to potentials with strength $\lambda > 0$.
The right-hand side of (\ref{pde1}) represents the local alignment forces with $\beta>0$ and the diffusion term in velocity with $\sigma>0$. It denotes the nonlinear
Fokker–Planck operator in \cite{villani2002review} given by
\begin{align*}
\nabla_v\cdot[\beta(v-u)f+\sigma\nabla_vf] = \sigma\nabla_v\cdot\Big(f\nabla_v \log\frac{f}{M_u}\Big)
\end{align*}
with the local Maxwellian
\begin{align*}
M_u:=\frac{\beta^{d/2}}{(2\pi\sigma)^{d/2}}\exp\Big(-\frac{\beta|u-v|^2}{2\sigma}\Big).
\end{align*}
Note that this local alignment force had been introduced in \cite{karper2013existence} for swarming models. In fact, it can be understood as the localized version of the nonlinear damping term introduced by \cite{motsch2011new} as a suitable normalization of the C-S model \cite{cucker2007emergent}. And it is also a nonlinear damping relaxation towards the local velocity used in classical kinetic theory \cite{cercignani2013mathematical, villani2002review}.

The kinetic C-S equation, pivotal in describing the motion of self-propelled particles such as bird flocking, fish schooling, and phenomena across biology, the internet, and sociology, was first introduced by Cucker and Smale in 2007 \cite{cucker2007emergent,cucker2007mathematics}. This foundational work has been further examined and applied in various studies, including those by \cite{grimm2005individual,huth1992simulation,lukeman2010inferring,motsch2011new,sneyd2001self,topaz2006nonlocal,viscido2004individual}. In 2008, Ha and Tadmor \cite{ha2008particle} expanded upon this by deriving a Vlasov-type mean-field model incorporating the C-S particles term and demonstrating its time-asymptotic flocking behavior for compactly supported initial data. Motsch and Tadmor, in 2011 \cite{motsch2011new}, proposed a modified C-S model featuring normalized and non-symmetric alignment. Although this variant was not the central focus of their study, it introduced an important perspective on alignment mechanisms. The exploration of the C-S model continued with Karper, Mellet, and Trivisa in 2013 \cite{karper2013existence}, who investigated the model's dynamics with strong local alignment, noise, self-propulsion, and friction, proving the existence of weak solutions. Their subsequent works in 2014 \cite{karper2014strong} and 2015 \cite{karper2015hydrodynamic} further validated the existence of weak solutions for the Motsch–Tadmor model and examined the singular limits of the C-S model under conditions of strong noise and alignment, respectively, illustrating convergence to the Euler-flocking system. Choi's 2016 study \cite{choi2016global} on the global well-posedness and asymptotic behavior of the Vlasov–Fokker–Planck equation with local alignment forces marked a significant advancement, demonstrating global existence and uniqueness of classical solutions. Lastly, in 2018, Figalli and Kang \cite{figalli2018rigorous} rigorously analyzed the singular limit of the C-S model with strong local alignment, showing its convergence to a pressureless Euler system with nonlocal flocking dissipation.
More recently, Carrillo and Choi explored the kinetic C-S flocking model, focusing on aspects of confinement, nonlocal interaction, and local alignment forces. In their 2020 study \cite{carrillo2020quantitative}, they not only discussed the existence of weak solutions to this model but also established quantitative bounds on the error between the kinetic equation and the aggregation equation through the limit of large friction. Building on this work, they further investigated a class of Vlasov-type equations with nonlocal forces in 2021 \cite{carrillo2021quantifying}, paying special attention to the existence of weak solutions for a continuity type equation as in (\ref{pde1}).

The primary goal of this paper is to deduce equation (\ref{pde1}) from a stochastic many-particle system. To this end, we consider a filtered probability space defined by $(\Omega, \mathcal{F}, (\mathcal{F})_{t\geq 0}, \mathbb{P})$ and introduce $\{B^i\}_{1\le i\le N}$, a collection of independent $\mathcal{F}_t$-Brownian motions. The dynamics of the system, which features moderate interactions among numerous particles, are described in relation to a set of parameters denoted by $\xi:= (\varepsilon, \delta, \nu)$. The model's formulation is as follows:
	\begin{align}\label{sde1}
        \left\{\begin{aligned}
		dX_{\xi,N}^{i} (t) =& V_{\xi,N}^{i}(t) dt,\qquad i\in[N]:=\{1\cdots\cdots N\},\\
		dV_{\xi,N}^{i}(t) = &\sqrt{2\sigma}dB^i(t)-\gamma V_{\xi,N}^{i}(t)dt- \beta \Big(V_{\xi,N}^{i}(t)- u_\xi\big(X_{\xi,N}^{i}(t)\big)\Big)dt \\
        &- \lambda\Big(\nabla_x V (X_{\xi,N}^{i}(t)) + \frac{1}{N}\sum_{i\ne j}^N \nabla_x W_\varepsilon\big(X_{\xi,N}^{i}(t) - X_{\xi,N}^{j}(t)\big)\Big)dt
        ,\end{aligned}\right.
	\end{align}
where $X_{\xi,N}^{i}, V_{\xi,N}^{i}$ denote the position and velocity of the $i$th particles at time $t$, and we assume the initial data $\{ (X_{\xi,N}^{i}(0), V_{\xi,N}^{i}(0))\}_{i=1}^N$ are i.i.d. with the common distribution $f_0$. Here $W_\varepsilon(x)$ is the regularized function of $W(x)$, which is given by
$$W_\varepsilon(x)= C_d(\varepsilon + |x|^2)^{-\frac{d-2}{2}},\qquad C_d=\frac{1}{(d-2)|B(0,1)|},$$ 
and we can directly compute
\begin{align}\label{W}
\|\nabla W_\varepsilon\|_{L^\infty(\R^d)}\le C\varepsilon^{-\frac{d}{2}},\qquad \|D^2 W_\varepsilon\|_{L^\infty(\R^d)}\le C\varepsilon^{-\frac{d+2}{2}}.
\end{align}
The approximated local particle's velocity is
	\begin{eqnarray}\label{u1}
    	u_\xi(X_{\xi,N}^{i}(t) ) = \frac{\frac{1}{N}\sum_{j=1}^N V_{\xi,N}^{j}(t)\phi_2^\delta(V_{\xi,N}^{j}(t))\cdot\phi_1^\varepsilon(X_{\xi,N}^{i}(t) - X_{\xi,N}^{j}(t) )}{\frac{1}{N}\sum_{j=1}^N \phi_1^\varepsilon(X_{\xi,N}^{i}(t) - X_{\xi,N}^{j}(t) ) + \nu}.
    \end{eqnarray}
For each $\varepsilon>0$, $\phi_1$ is the standard mollifier,  the function $\phi_{1}^\varepsilon(x)\in C_c^\infty(\R^d)$ satisfying $\int_{\mathbb{R}^d} \phi_{1}^\varepsilon dx=1,\ spt(\phi_{1}^\varepsilon)\subset B(0,\varepsilon)$. For $\phi_2(v)\in C^2[0,+\infty)$ satisfying
\begin{eqnarray*}
\phi_2(v) =
\begin{cases}
1,& v\le 1,\\
0,& v\ge 2,
\end{cases}
\quad 0\le\phi_2\le1,\ |\phi_2'|\le C\ \mbox{and}\ |\phi_2''|\le C,
\end{eqnarray*}
and taking $\phi_2^\delta(v) = \phi_2(\delta|v|),\ v\in\mathbb{R}^d$,  we have
\begin{align*}
&\parallel v\phi_2^\delta(v) \parallel_{L^\infty(\R^d)} \le 2/\delta,\qquad
\parallel\nabla\big(v\phi_2^\delta(v)\big) \parallel_{L^\infty(\R^d)} \le C.
\end{align*}
We will demonstrate that the proposed particle system (\ref{sde1}) accurately approximates the kinetic equation (\ref{pde1}) in the limit as $N \to \infty$ and $\xi \to 0$.

Firstly, we study the large particle limit. Specifically, with $\xi > 0$ held fixed, we illustrate how the system (\ref{sde1}) approaches an intermediate stochastic system in the limit of $N \to \infty$
\begin{align}\label{sde2}
\left\{\begin{aligned}
d\overline{X}_{\xi}^{i}(t) =& \overline{V}_{\xi}^{i}(t) dt,\\
d\overline{V}_{\xi}^{i}(t) =& \sqrt{2\sigma}dB^i(t)-\gamma \overline{V}_{\xi}^{i}(t)dt
-\beta \Big(\overline{V}_{\xi}^{i}(t)- \overline{u}_\xi\big(\overline{X}_{\xi}^{i}(t)\big)\Big)dt\\
&- \lambda\Big(\nabla_x V(\overline{X}_{\xi}^{i}(t)) + \nabla_x W_\varepsilon* \rho_\xi\big(\overline{X}_{\xi}^{i}(t)\big)\Big)dt,
\end{aligned}\right.
\end{align}
where $\rho_\xi(x)=\int_{\R^d}f_\xi(x,dv)$ and $f_\xi(x,v,t)$ is  the probability density of $(\overline{X}_{\xi}^{i}(t),\overline{V}_{\xi}^{i}(t))$.
The initial data $\{(\overline{X}_{\xi}^i(0), \overline{V}_{\xi}^i(0))\}_{i=1}^N$ have the same distribution as the initial data of (\ref{sde1}). Here $\overline{u}_\xi$ is given by
	\begin{eqnarray}\label{u2}
	\overline{u}_\xi(\overline{X}_{\xi}^{i}(t)) = \frac{\int_{\mathbb{R}^d}v\phi_2^\delta(v)\cdot\phi_1^\varepsilon*f_\xi(\overline{X}_{\xi}^{i}(t)) dv}{\phi_1^\varepsilon * \rho_\xi(\overline{X}_{\xi}^{i}(t)) + \nu}.
    \end{eqnarray}
System \eqref{sde2} is uncoupled, since $\overline{X}_{\xi}^{i}$ depends on $N$ only through the initial datum. 
In Section 3, we present a key contribution of this paper: an error estimate of the mean-field limit through comparing the solutions of the moderately interacting many-particle system, $\{(X_{\xi,N}^i(t), V_{\xi,N}^i(t))_{0 \le t \le T}\}_{i=1}^N$ from (\ref{sde1}), with those of the intermediate system, $\{(\overline{X}_{\xi}^i(t), \overline{V}_{\xi}^i(t))_{0 \le t \le T}\}_{i=1}^N$ from (\ref{sde2}):
\begin{align*}
    \sup\limits_{0<t<T}\sup \limits_{i=1\dots N}  \mathbb{E}\left[|X_{\xi,N}^i(t) - \overline{X}_{\xi}^i(t)|^2+|V_{\xi,N}^i(t) - \overline{V}_{\xi}^i(t)|^2\right] &\leq
\frac{C T^3\ln (N^\alpha)}{N}\big(1+ T^3N^{T^3\alpha }\ln (N^{\alpha})\big)\,,
    \end{align*}
where for fixed $0<\alpha\le 1$,
$\ln (N^\alpha)\sim 1/(\delta^2\nu^4\varepsilon^{4d+2}) $, namely we have chosen $\xi$ to be logarithmically with respect to the number of particles.

The rigorous derivation of mean-field equations from their underlying many-particle systems has been a subject of investigation since the 1980s, with seminal contributions from Sznitman \cite{sznitmantopics,sznitman1984nonlinear}.
To address the singularity presented in the interaction force $W$, the use of a cut-off parameter has become a common approach, as highlighted in works such as \cite{boers2016mean,huang2017error,huang2020mean,lazarovici2017mean}. For the treatment of the local alignment term $u$, a strategy known as moderate interaction is applied within $u_\xi$, resulting in interactions that are more localized compared to those found in the mean-field regime. A closely related work by Oelschläger \cite{oelschlager1984martingale} demonstrated that interacting stochastic particle systems weakly converge to a deterministic nonlinear process as the number of particles increases. He further elucidated the derivation of the porous-medium equation from the limit dynamics of a large interacting particle system \cite{oelschlager1990large}, with subsequent applications found in \cite{jourdain1998propagation,philipowski2007interacting}. In more recent developments, Chen, Göttilich and Knapp \cite{chen2018modeling} applied Oelschläger's approach to rigorously derive error estimates for solutions to stochastic particle systems and their diffusion-aggregation equation limits. Additionally, Chen, Daus and Jüngel \cite{chen2019rigorous} established the mean-field limit for weakly interacting stochastic many-particle systems across multiple population species in the entire space. Finally, Chen, Holzinger, Jüngel and Zamponi \cite{chen2022analysis} executed a mean-field-type limit for stochastic moderately interacting many-particle systems with singular Riesz potential, leading to the emergence of nonlocal porous-medium equations.

Furthermore, for a fixed $\xi > 0$, applying Itô's formula to the intermediate stochastic system yields a regularized version of the equation (\ref{pde1}) for $f_\xi$-the probability density function of the solutions $\overline{X}_{\xi}^i(t)$ to (\ref{sde2}):
	\begin{eqnarray}\label{repde1}
    	&\partial_t f_\xi+v\cdot\nabla_x f_\xi- \nabla_v\cdot[\big(\gamma v+\lambda(\nabla_x V +\nabla_x W_\varepsilon*\rho_\xi)\big)f_\xi]=\nabla_v\cdot[\beta(v-u_\xi)f_\xi+\sigma\nabla_vf_\xi], \\
		&f_\xi(x,v,0) = f_0, \quad (x,v)\in \mathbb R^d\times \mathbb R^d,\nonumber
	\end{eqnarray}
	where
\begin{align}\label{u0}
u_\xi(x,t) = \frac{\int_{\mathbb{R}^d}v\phi_2^\delta(v)\cdot\phi_1^\varepsilon * f_\xi dv}{\nu+\int_{\mathbb{R}^d}\phi_1^\varepsilon * f_\xi dv}.
\end{align} 

Lastly, by taking the limit as $\xi \to 0$ in the regularized equation (\ref{repde1}), we establish that the limit function $f$, derived from $f_\xi$, fulfills the equation (\ref{pde1}). In Section 2, we provide proofs for the existence of solutions to both (\ref{pde1}) and (\ref{repde1}). Our proof methodology draws upon the approaches found in the works of Carrillo et al. \cite{carrillo2020quantitative, carrillo2021quantifying} and Karper et al. \cite{karper2013existence}, yet it introduces a novel perspective through the application of our different regularized equation.

This paper is organized as follows: Section 2 details the establishment of global existence for weak solutions to the kinetic equation (\ref{pde1}) and its regularized counterpart, alongside deriving the convergence relations between their solutions. In Section 3, we demonstrate the existence and uniqueness for the moderately interacting particle system (\ref{sde1}) and the intermediate system (\ref{sde2}). Moreover, this section highlights the main contribution of our study: mean-field limit error estimates between the solutions of the moderately interacting many-particle system and the intermediate system, quantified in terms of expectation.

\section{EXISTENCE OF WEAK SOLUTIONS}

This section focuses on establishing the global existence of weak solutions for the kinetic equation (\ref{pde1}) and its associated regularized problem, drawing inspiration from the methodologies detailed in \cite{carrillo2021quantifying} and \cite{karper2013existence}. In the first subsection, we address the global existence of weak solutions to the regularized problem \eqref{repde1}. Subsequently, the subsection 2.2 demonstrates the derivation of weak solutions for the kinetic equation (\ref{pde1}) by considering the limit as $\xi\to 0$.

We first show the following
lemma, which is quoted from \cite[Lemma 5.3, Lemma 5.4]{carrillo2021quantifying} and \cite[Lemma 2.7]{karper2013existence}.
\begin{lem}\label{lem2.1}
 Let $\{f^n\}_{n\in \mathbb{N}_+}$ and $\{G^n\}_{n\in \mathbb{N}_+}$ be bounded in $L^p_{loc}([0,T]\times\mathbb{R}^{d}\times\mathbb{R}^{d})$ with $1<p<\infty$. If $f^n$ and $G^n$ satisfy
\begin{align}\label{local}
f_t^n + v\cdot\nabla_xf^n= \partial_v G^n,\qquad
f^n|_{t=0} = f^0 \in L^p(\mathbb{R}^{d}\times\mathbb{R}^{d}),
\end{align} and assume that for $r\ge2$,
$$\sup\limits_{n\in\mathbb{N}_+}\| f^n\|_{L^\infty([0,T]\times\mathbb{R}^{d}\times\mathbb{R}^{d})} + \sup\limits_{n\in\mathbb{N}_+}\| (|v|^r+|x|^2)f^n\|_{L^\infty(0,T;L^1(\mathbb{R}^{d}\times\mathbb{R}^{d}))}<\infty.$$
Then for any $\varphi(v)$ such that $|\varphi(v)|\le c(1+|v|)$, the sequence $$ \Big\{ \int_{\mathbb{R}^{d}} f^n\varphi(v)dv \Big\}_{n\in \mathbb{N}_+}$$ is relatively compact in $L^q([0,T]\times\mathbb{R}^{d})$ for any $q\in(1,\frac{d+r}{d+1})$.
\end{lem}
To clarify the limit of $\Big\{ \int_{\mathbb{R}^{d}} f^n\varphi(v)dv \Big\}_{n\in \mathbb{N}_+}$, we further give the following lemma.
\begin{lem}\label{lem2.2}
 For any $r\geq 2$, let $\{f^n\}_{n\in \mathbb{N}_+}$ be a sequence satisfying
 $$\sup\limits_{n\in\mathbb{N}_+}\| f^n\|_{L^\infty([0,T]\times\mathbb{R}^{d}\times\mathbb{R}^{d})} + \sup\limits_{n\in\mathbb{N}_+}\| (|v|^r+|x|^2)f^n\|_{L^\infty(0,T;L^1(\mathbb{R}^{d}\times\mathbb{R}^{d}))}<\infty$$ and
 $$f^n\rightharpoonup f \qquad in\ L^\infty(0,T;L^p(\mathbb{R}^{d}\times\mathbb{R}^{d}))\quad \forall p\in (1,\infty]\,.$$
 Assume that for  any $\varphi(v)$ satisfying $|\varphi(v)|\le c(1+|v|)$, the sequence  $\Big\{ \int_{\mathbb{R}^{d}} f^n\varphi(v)dv \Big\}_{n\in \mathbb{N}_+}$ converges strongly to some $h\in L^q([0,T]\times\mathbb{R}^{d})$ for any $q\in(1,\frac{d+r}{d+1})$. Then it holds that
$$  \int_{\mathbb{R}^{d}} f^n\varphi(v)dv  \to  \int_{\mathbb{R}^{d}} f\varphi(v)dv\qquad  in \ L^q([0,T]\times\mathbb{R}^{d})\quad as\ n\to\infty\,.$$

\end{lem}
\begin{proof}
Firstly, according to the assumption, it holds that for $q\in(1,\frac{d+r}{d+1})$, 
\begin{equation}\label{weakconve}
f^n \rightharpoonup  f \quad \mbox{ in }L^q([0,T]\times\mathbb{R}^{2d})\,.
\end{equation}
Then we claim that
\begin{equation}
\int_{\R^d}\varphi(v)f^ndv \rightharpoonup  \int_{\R^d}\varphi(v)fdv \quad \mbox{ in }L^q([0,T]\times\mathbb{R}^{d})\,.
\end{equation}

To prove the claim, we consider the text function $\Psi(t,x)\in L^{q'} ([0,T]\times\R^{d})$ with $\frac{1}{q}+\frac{1}{q'}=1$. Then for any $R_\varepsilon>0$ it holds that 
\begin{align*}
&\Big|\int_0^T\int_{\mathbb{R}^{d}}\int_{\mathbb{R}^{d}}f^n\varphi(v)dv\Psi(t,x)dxdt-\int_0^T\int_{\mathbb{R}^{d}}\int_{\mathbb{R}^{d}}f\varphi(v)dv\Psi(t,x)dxdt\Big|\\
\le &\Big|\int_0^T\int_{\mathbb{R}^{d}}\int_{|v|\le R_\varepsilon}( f^n- f)\varphi(v)\Psi(t,x)dvdxdt\Big|+\Big|\int_0^T\int_{\mathbb{R}^{d}}\int_{|v|\ge R_\varepsilon}( f^n- f)\varphi(v)\Psi(t,x)dvdxdt\Big|.
\end{align*}

On the one hand, for any $\varepsilon>0$, there exists some $R_\varepsilon$ sufficiently large such that
\begin{align}\label{term1}
&\Big|\int_0^T\int_{\mathbb{R}^{d}}\int_{|v|\ge R_\varepsilon}( f^n- f)\varphi(v)\Psi(t,x)dvdxdt\Big|^q\nonumber\\
\le
&C\Big\|\Psi(t,x)\Big\|_{L^{q'}([0,T]\times\R^{d})}^q
\Big\|\int_{|v|\ge R_\varepsilon}\varphi(v)(f^n-f)dv\Big\|_{L^q([0,T]\times\R^{d})}^q\nonumber\\
\le&C\int_0^T\int_{\R^d}\Big|\int_{|v|\ge R_\varepsilon}c(1+|v|)|f^n-f|dv\Big|^qdxdt\nonumber\\
\le&C\int_0^T\int_{\R^d}\Big|\int_{|v|\ge R_\varepsilon}(c(1+|v|))^\frac{r}{q}|f^n-f|^\frac{1}{q}\frac{|f^n-f|^\frac{1}{q'}}{(c(1+|v|))^{\frac{r}{q}-1}}dv\Big|^qdxdt\nonumber\\
\le&C\int_0^T\int_{\R^d}\int_{|v|\ge R_\varepsilon}c^r(1+|v|)^r |f^n-f|dv\Big(\int_{|v|\ge R_\varepsilon}\frac{|f^n-f|}{(c(1+|v|))^{\frac{rq'}{q}-q'}}dv\Big)^\frac{q}{q'}dxdt\nonumber\\
\le&\frac{C\|f^n-f\|^{\frac{q}{q'}}_{L^\infty([0,T]\times\mathbb{R}^{2d})}}{R_\varepsilon^{r-q-\frac{dq}{q'}}}\int_0^T\int_{\R^d}\int_{\R^{d}}(1+|v|)^r |f^n-f| dvdxdt\leq \varepsilon,
\end{align}
where $\frac{rq'}{q}-q'>d$ is equivalent to $1<q<\frac{d+r}{d+1}$. 

On the other hand, for fixed $R_\varepsilon>0$, it is easy to check that
\begin{align*}
&\int_0^T\int_{\R^d}\int_{\R^d}|\varphi(v)\textbf{I}_{|v|\leq R_\varepsilon}\Psi(t,x)|^{q'}dvdxdt \notag \\
\leq& \int_0^T \int_{\R^d}c^{q'}(1+|v|)^{q'}\textbf{I}_{|v|\leq R_\varepsilon} dv\int_{\R^d}|\Psi(t,x)|^{q'}dxdt \notag \\
\leq& C(R_\varepsilon,c,d)\|\Psi\|^{q'}_{L^{q'} ([0,T]\times\R^{d})}<\infty.
    \end{align*}
This means that $\varphi(v)\textbf{I}_{|v|\leq R_\varepsilon}\Psi(t,x)$ can be seen as a test function in $L^{q'} ([0,T]\times\R^{d})$, which according to \eqref{weakconve} leads to
\begin{equation}
\Big|\int_0^T\int_{\mathbb{R}^{d}}\int_{|v|\leq R_\varepsilon}( f^n- f)\varphi(v)\Psi(t,x)dvdxdt\Big|\to 0\quad\mbox{ as }n\to\infty.
\end{equation}

This combining with \eqref{term1} concludes that
\begin{align*}
\int_{\R^d}f^n\varphi(v)dv\rightharpoonup \int_{\R^d}f\varphi(v)dv \qquad \mbox{as} \ n\to\infty\quad \mbox{in} \ L^q([0,T]\times\R^{d})
\end{align*}
for any $q\in(1,\frac{d+r}{d+1})$. 

Notice that by relative compactness assumption there exists some $h \in L^q([0,T]\times\mathbb{R}^{d})$ such that $\{ \int_{\mathbb{R}^{d}} f^n\varphi(v)dv \}_{n\in \mathbb{N}_+}$ converges to $h$ strongly. Due to the uniqueness of the limit we have
\begin{align*}
\int_{\R^d}f^n\varphi(v)dv\to h = \int_{\R^d}f\varphi(v)dv \qquad \mbox{as} \ n\to\infty\quad \mbox{in} \ L^q([0,T]\times\R^{d}).
\end{align*}

\end{proof}
\subsection{Solvability of the regularized problem}
As already noted in the introduction, the regularized problem for $f_\xi$ is
\begin{eqnarray}\label{pde2}
    	&\partial_t f_\xi+v\cdot\nabla_x f_\xi- \nabla_v\cdot[\big(\gamma v+\lambda(\nabla_x V +\nabla_x W_\varepsilon*\rho_\xi)\big)f_\xi]=\nabla_v\cdot[\beta(v-u_\xi)f_\xi+\sigma\nabla_vf_\xi], \\
		&f_\xi(x,v,0) = f_0, \quad (x,v)\in \mathbb R^d\times \mathbb R^d,\nonumber
	\end{eqnarray}
	where $u_\xi$ is defined in (\ref{u0}) and has the following form
\begin{align*}
u_\xi(x,t) = \frac{\int_{\mathbb{R}^d}v\phi_2^\delta(v)\cdot\phi_1^\varepsilon* f_\xi dv}{\nu+\int_{\mathbb{R}^d}\phi_1^\varepsilon * f_\xi dv}.
\end{align*}

The purpose of this subsection is to prove existence of weak solutions for regularized problem \eqref{pde2}. The main theorem is as follows.
\begin{thm}\label{thm2.1}
Let $T>0$. Suppose that $f_0$ satisfies
$$ f_0 \in L_+^1\cap L^\infty(\mathbb R^d\times\mathbb{R}^d) \quad and\quad (|v|^2+V+W*\rho_0)f_0 \in L^1(\mathbb R^d\times\mathbb{R}^d).$$
Then for any $\lambda,\beta,\gamma >0$, there exists a weak solution $f_\xi$, $\xi=(\varepsilon,\delta,\nu)$ of Eq.(\ref{pde2}) satisfying
\begin{eqnarray*}
&f_\xi\in L^\infty(0,T;L^p(\mathbb{R}^d\times\R^d)),\quad &\forall p\in[1,\infty],\\
&\rho_\xi\in L^\infty(0,T;L^{q_1}(\mathbb{R}^d)),\quad &\forall q_1\in[1,(d+2)/d),\\
& j_\xi\in L^\infty(0,T; L^{q_2}(\mathbb{R}^d)),\quad &\forall q_2\in[1,(d+2)/(d+1)).
\end{eqnarray*}
And $f_\xi$ satisfies the following entropy inequality
\begin{align}\label{entropy}
&\int_{\mathbb{R}^{2d}}\Big(\frac{|v|^2}{2}+\frac{|x|^2}{2}+V+\sigma |logf_\xi|\Big)f_\xi dxdv 
+\frac{1}{2}\int_{\mathbb{R}^{2d}}W_\varepsilon(x-y)\rho_\xi(x)\rho_\xi(y)dxdy \nonumber\\
\le&-\int_0^T\int_{\mathbb{R}^{2d}}\frac{1}{f_\xi}\Big|\sigma\nabla_vf_\xi-(v-u_\xi)f_\xi\Big|^2dxdvdt+
 \int_{\mathbb{R}^{2d}}\Big(\frac{|v|^2}{2}+\frac{|x|^2}{2}+V+\sigma |logf_0|\Big)f_0 dxdv \nonumber\\
&+\frac{1}{2}\int_{\mathbb{R}^{2d}}W_\varepsilon(x-y)\rho_0(x)\rho_0(y)dxdy + \int_0^T\int_{\mathbb{R}^{2d}}(|v|^2+|x|^2)f_\xi dxdvdt+C,
\end{align}
where $C = C(T)$ is a positive constant depending on $\gamma,\beta,\lambda,T$ and $\|f_0\|_{L^1}$.
\end{thm}
\begin{proof}

Step 1: Establish the regularized equation of (\ref{pde2}) 
with
respect to regularization parameter $\eta:=(R, \zeta, \xi)$ as follows
\begin{align}\label{pde3}
\partial_t f^\eta +v\cdot\nabla_x f^\eta- \nabla_v\cdot[\big(\gamma &v+\lambda(\nabla_x V^R +\nabla_x W_\varepsilon*\rho^\eta)\big)f^\eta]\nonumber\\
&= \nabla_v\cdot[\beta(v-\chi_\zeta(u^\eta))f^\eta+\sigma\nabla_vf^\eta],
\end{align}
subject to initial data
\begin{align*}
&f_0^\eta=f^\eta(x,v,0):=\begin{cases}
		f_0,& |v|\le\zeta,\\
		0,& |v|>\zeta.
		\end{cases}
\end{align*}
Where $\rho^\eta = \rho^\eta(x,t)$ and $u^\eta = u^\eta(x,t)$ are given by
$$\rho^\eta:=\int_{\mathbb{R}^d} f^\eta dv, \quad
u^\eta := \frac{k^\eta}{\nu+\phi_1^\varepsilon *\rho^\eta},
$$
and $$k^\eta(x,t) =\int_{\mathbb{R}^d}v\phi_2^\delta(v)\cdot\phi_1^\varepsilon* f^\eta dv,
$$ 
for convenience, we assume that
$$j^\eta(x,t):=\int_{\mathbb{R}^d} vf^\eta dv.$$
Moreover, $\chi_\zeta$, $V^R$ are defined by
\begin{eqnarray*}
\chi_\zeta(u) =
\begin{cases}
u,& |u|\le\zeta,\\
0,& |u|>\zeta,
\end{cases}
\qquad V^R(x):=V(x)M(\frac{x}{R}),
\end{eqnarray*}
where $M(x) \in \mathcal{C}_c^\infty(\mathbb{  R}^d)$ is a smooth function defined by
\begin{eqnarray*}
M(x) =
\begin{cases}
1,& |x|< 1,\\
0,& |x|>2,
\end{cases}
\ \mbox{and} \quad
0\le M(x)\le1. 
\end{eqnarray*}
\\
Step 2:  In this step, we prove existence of weak solutions to (\ref{pde3}).

Firstly, similarly as in Ref.\cite{karper2013existence}, we partially linearize (\ref{pde3}) as follows
	\begin{align}\label{pde4}
    	\partial_t f^\eta +v\cdot\nabla_x f^\eta- \nabla_v\cdot[\big(\gamma& v+\lambda(\nabla_x V^R +\nabla_x W_\varepsilon*\rho^\eta)\big)f^\eta]\nonumber\\
		&=\nabla_v\cdot[\beta(v-\chi_\zeta(\widetilde{u}))f^\eta+\sigma\nabla_vf^\eta],
	\end{align}
where $\widetilde{u}$ is in $\mathcal{S} = L^2([0,T]\times\mathbb{R}^d)$. Since for any fixed $\varepsilon>0$, $\nabla_x W_\varepsilon$ is bounded and Lipschitz
continuous from (\ref{W}), existence of weak solutions $f^\eta$ to (\ref{pde4}) comes from almost the same
argument in Ref.\cite[Theorem 6.3]{karper2013existence}, and $f^\eta$ satisfies
\begin{align}\label{f1}
\parallel f^\eta\parallel_{L^\infty(0,T;L^p(\mathbb{R}^{2d}))}
+
\sigma\parallel\nabla_v (f^\eta)^\frac{p}{2}\parallel_{L^2([0,T]\times\mathbb{R}^{2d})}^2
\le e^{CT}
\parallel f_0^\eta\parallel_{L^p(\mathbb{R}^{2d})}
\end{align}
for any $p\in[1,\infty)$, where $C$ is a constant independent of $p, \eta$. In particular, we have for $t\in[0,T]$,
\begin{align}\label{f_inf}
\|f^\eta(\cdot,\cdot,t)\|_{L^1(\R^{2d})}\le\|f_0\|_{L^1(\R^{2d})},\qquad \|f^\eta(\cdot,\cdot,t)\|_{L^\infty(\R^{2d})}\le e^{CT}\|f_0\|_{L^\infty(\R^{2d})}.
\end{align}
Its velocity moment satisfies the
following boundedness estimate from Ref.\cite[Lemma 5.1]{carrillo2021quantifying},
\begin{align}\label{moment}
\sup\limits_{t\in(0,T)}\int_{\mathbb{R}^{d}}\int_{\mathbb{R}^{d}}|v|^N f^\eta dvdx\le C \qquad\forall\  N\ge 0,
\end{align}
where $C$ is a constant independent of $R, \zeta$.
\\

Then, we define the mapping $\mathcal{T}:\mathcal{S}\to\mathcal{S}$, where $\mathcal{S} = L^q([0,T]\times\mathbb{R}^d),q\in[1,\frac{d+2}{d+1})$ by
\begin{align}\label{T}
\widetilde{u}\to \mathcal{T}(\widetilde{u}):= u^\eta(x,t) = \frac{k^\eta}{\nu+\phi_1^\varepsilon * \rho^\eta }.
\end{align}
In the following, we show $\mathcal{T}$ is well-defined. Noticing the definition of $\phi_2^\delta$ and using the $\mathrm{H\ddot{o}lder}$ inequality, it holds for each $\delta>0$,
\begin{align*}
\|k^\eta\|_{L^q([0,T]\times{\mathbb{R}^{d}})}^q
&=\int_0^T\int_{\mathbb{R}^d} \Big|\int_{|v|<\frac{2}{\delta}}v\phi_1^\varepsilon * f^\eta dv\Big|^q dxdt\\
&\le\int_0^T\int_{\mathbb{R}^d}\Big(\int_{|v|<\frac{2}{\delta}} 1 dv \Big)^{q-1}\int_{|v|<\frac{2}{\delta}}|\phi_1^\varepsilon * vf^\eta|^q dv dxdt\\
&\le \frac{C(d)}{\delta^{d(q-1)}} \int_0^T\int_{\mathbb{R}^d} \int_{|v|<\frac{2}{\delta}}\Big|\int_{\mathbb{R}^d}\phi_1^\varepsilon(x-y) \cdot vf^\eta(y,v)dy\Big|^q dvdxdt.
\end{align*}
Moreover, we have
\begin{align}\label{k1}
\|k^\eta\|_{L^q([0,T]\times{\mathbb{R}^{d}})}^q
&\le \frac{C(d)}{\delta^{d(q-1)}} \int_0^T\int_{\mathbb{R}^d} \int_{|v|<\frac{2}{\delta}}\Big(\int_{\mathbb{R}^d}\phi_1^\varepsilon(x-y)dy\int_{\mathbb{R}^d} \phi_1^\varepsilon(x-y)(vf^\eta(y,v))^qdy\Big)dvdxdt\nonumber\\
&\le \frac{C(d)}{\delta^{d(q-1)}}
\int_0^T\int_{\mathbb{R}^{d}}
\Big(\int_{\mathbb{R}^{d}}\phi_1^\varepsilon(x-y)dx\int_{|v|<\frac{2}{\delta}} |v|^qf^\eta(y,v)^qdv\Big)dydt\nonumber\\
&\le \frac{C(d)}{\delta^{d(q-1)}}\| f^\eta\|_{L^\infty([0,T]\times\R^{2d})}^{q-1}\int_0^T\int_{\mathbb{R}^{d}} \int_{\mathbb{R}^{d}}|v|^qf^\eta(y,v)dvdydt.
\end{align}
By the velocity moment estimate (\ref{moment}), we obtain the boundedness of $k^\eta$ in $ L^q([0,T]\times\mathbb{R}^{d})$. Thus we have
\begin{align}\label{u}
\|u^\eta\|_{L^q([0,T]\times {\mathbb{R}^{d}})} &=\Big\|\frac{k^\eta}{\nu+\phi_1^\varepsilon*\rho^\eta}\Big\|_{L^q([0,T]\times {\mathbb{R}^{d}})}
\le\frac{1}{\nu}\|k^\eta\|_{L^q([0,T]\times {\mathbb{R}^{d}})}<C\quad \forall t\in[0,T],
\end{align}
where $C$ is a constant independent of $R, \zeta$.
\\

Next, we prove the compactness of $\mathcal{T}$. Similar to \eqref{T}, define ${u^\eta_m}$ by
$$\widetilde{u}_m\to \mathcal{T}(\widetilde{u}_m) := u^\eta_m(x,t) = \frac{\int_{\mathbb{R}^d}v\phi_2^\delta(v)\cdot\phi_1^\varepsilon * f^\eta_m dv}{\nu+\int_{\mathbb{R}^d}\phi_1^\varepsilon * f^\eta_m dv}\quad \forall m\in \mathbb{N}_+,$$ 
where $\{\tilde{u}_m\}_{m\in \mathbb{N}_+}$ is a bounded and compact sequence in $\mathcal{S} = L^q([0,T]\times\mathbb{R}^d)$.
It is similar to
discuss the compactness of $\mathcal{T}$ i.e.
the convergence of $\{u^\eta_m\}_{m\in \mathbb{N}_+}$ in  Ref.\cite[Lemma 3.5]{karper2013existence}, we present here for readers'convenience. Set
\begin{align}\label{G}
&G_m := \Big(\gamma v+\lambda(\nabla_x V^R +\nabla_x W_\varepsilon*\rho^\eta_m)+ \beta(v-\chi_\zeta(u^\eta_m))\Big)f^\eta_m+\sigma\nabla_vf^\eta_m.
\end{align} 
Now we verify that  $\{G_m\}_{m\in \mathbb{N}_+}$ is bounded in $L^\infty(0,T;L_{loc}^p(\R^{2d})),p\in[1,2]$. In fact, the $\mathrm{H\ddot{o}lder}$ inequality gives that for $t\in[0,T]$,
\begin{align*}
\|G_m\|_{L^\infty(0,T;L_{loc}^p(\R^{2d}))} \le &(\gamma +\beta) \|vf^\eta_m\|_{L^\infty(0,T;L^p(\R^{2d}))}+\sigma\|\nabla_vf^\eta_m\|_{L^\infty(0,T;L^p(\R^{2d}))}\\
&+\big(C\lambda(1+\frac{1}{R} )+C\lambda\varepsilon^{-\frac{d}{2}}+ \beta\zeta\big)\|f^\eta_m\|_{L^\infty(0,T;L^p(\R^{2d}))},
\end{align*}
where we use that in a bounded
region $\mathcal{D}\subset\R^d\times\R^d$, 
\begin{align*}
|\nabla_x V^R|&= \Big|\nabla V(x)M(\frac{x}{R}) +\frac{1}{R}V(x)\nabla M(\frac{x}{R})\Big|\\
&\le |\nabla V(x)|+ \frac{\parallel\nabla M\parallel_{L^\infty(\mathcal{D})}}{R}|V(x)|\le C(1+\frac{1}{R}),
\end{align*}
and for $t\in[0,T]$,
\begin{align*}
&\| \nabla_x W_\varepsilon*\rho^\eta_m\|_{L^\infty(\R^d)}\le\varepsilon^{-\frac{d}{2}}\|\rho^\eta_m\|_{L^1({\R^{d}})}
\le C \varepsilon^{-\frac{d}{2}},\qquad
\|\chi_\zeta(\tilde{u}_m)\|_{L^\infty({\R^{d}})}\le\zeta.
\end{align*}
It remains to bound the term involving $\nabla_v f^\eta_m$. By the $\mathrm{H\ddot{o}lder}$ inequality , we get that $\forall p\in[1,2)$,
\begin{align*}
\int_{\R^{d}}\int_{\R^{d}}|\nabla_vf^\eta_m|^pdvdx
&=
\int_{\R^{d}}\int_{\R^{d}}(f^\eta_m)^\frac{p}{2}(f^\eta_m)^{-\frac{p}{2}} |\nabla_v f^\eta_m|^p dvdx\nonumber\\
&\le 
\|(f^\eta_m)^\frac{p}{2}\|_{L^\frac{2}{2-p}(\R^{2d})}\Big(\int_{\R^{d}}\int_{\R^{d}}\frac{1}{f^\eta_m} |\nabla_v f^\eta_m|^2 dvdx\Big)^\frac{p}{2},\ 
\end{align*}
and when $p=2$,
\begin{align*}
\int_{\R^{d}}\int_{\R^{d}}|\nabla_vf^\eta_m|^2dvdx
&\le 
\|f^\eta_m\|_{L^\infty(\R^{2d})}\int_{\R^{d}}\int_{\R^{d}}\frac{1}{f^\eta_m} |\nabla_v f^\eta_m|^2 dvdx.
\end{align*}
Here (\ref{f1}) provides the following estimate 
\begin{align*}
\int_{\R^{d}}\int_{\R^{d}} \frac{1}{f^\eta_m}|\nabla_vf^\eta_m|^2dvdx
=\int_{\R^{d}}\int_{\R^{d}}4 |\nabla_v(f^\eta_m)^\frac{1}{2}|^2dvdx\le C,
\end{align*} 
where $C$ is a constant independent of $R, \zeta$. Thus, for $t\in[0,T]$,
\begin{align}\label{td}
\int_{\R^{d}}\int_{\R^{d}}|\nabla_vf^\eta_m|^pdvdx
\le C\qquad \forall p\in[1,2]. 
\end{align}
By (\ref{f1}), (\ref{moment}) and (\ref{td}), we have  
\begin{align}\label{Gm}
\|G_m\|_{L^\infty(0,T;L^p_{loc}(\mathbb{R}^{2d}))}\le C\qquad\ \forall  p\in[1,2].
\end{align}
By the uniform estimates of $f^\eta_m$,  $|v|f^\eta_m$ and $G_m$ in (\ref{f1}),  (\ref{moment}) and (\ref{Gm}), we know that 
there is a subsequence of $f^\eta_m$ (without relabeling for convenience) and $f^\eta\in L^\infty (0,T;L^p(\R^{2d}))$ such that as $m\to \infty$
\begin{equation*}
f^\eta_m \rightharpoonup  f^\eta \quad \mbox{ in }L^\infty (0,T;L^p(\R^{2d}))\qquad\ \forall  p\in[1,2].
\end{equation*}
Moreover, applying Lemma \ref{lem2.1} gives  
for any $\varphi(v)$ satisfying $|\varphi(v)|\le c(1+|v|)$, the sequence  $\Big\{ \int_{\mathbb{R}^{d}} f^\eta_m\varphi(v)dv \Big\}_{m\in \mathbb{N}_+}$ converges strongly to some $h\in L^q([0,T]\times\mathbb{R}^{d})$ for any $q\in(1,\frac{d+2}{d+1})$. Thus, in Lemma \ref{lem2.2} (with $\varphi(v) = 1$ and $\varphi(v)=v\phi_2^\delta(v)$), we show the following strong convergences in $L^q([0,T]\times\mathbb{R}^d)$,
\begin{align}\label{stc1}
&\int_{\R^d}f^\eta_m dv\to \int_{\R^d}f^\eta dv,\
\int_{\R^d}v\phi_2^\delta(v)f^\eta_m dv\to \int_{\R^d}v\phi_2^\delta(v)f^\eta dv \qquad \mbox{as}\ m\to \infty.
\end{align} 
Without relabeling for convenience, we also can get the following convergent subsequences
\begin{align*}
&\int_{\R^d}f^\eta_m dv\to \int_{\R^d}f^\eta dv,\
\int_{\R^d}v\phi_2^\delta(v)f^\eta_m dv\to \int_{\R^d}v\phi_2^\delta(v)f^\eta dv \qquad \ \mbox{as}\ m\to \infty\ \ \mbox{a.e.}.
\end{align*} 

Moreover, we give the convergence of $\{u^\eta_m\}_{m\in \mathbb{N}_+}$ up to a subsequence in $L^q((0,T)\times\mathbb{R}^d)$. Consider 
\begin{align}\label{up}
|u^\eta_m- u^\eta|
&=\Big|\frac{ k^\eta_m\big(\nu+\phi_1^\varepsilon* \rho^\eta \big)-k^\eta \big(\nu+ \phi_1^\varepsilon*\rho^\eta_m\big) }{\big(\nu+ \phi_1^\varepsilon*\rho^\eta_m \big)\big(\nu+ \phi_1^\varepsilon*\rho^\eta\big)}\Big|\nonumber\\
&= \Big|\frac{\nu\big(k^\eta_m- k^\eta\big) +  \phi_1^\varepsilon*\rho^\eta\big(k^\eta_m- k^\eta\big)-\phi_1^\varepsilon*(\rho^\eta_m -\rho^\eta\big) k^\eta}{(\nu+ \phi_1^\varepsilon*\rho^\eta_m)(\nu+ \phi_1^\varepsilon*\rho^\eta)}\Big|\nonumber\\
&\le \frac{2}{\nu}| k^\eta_m-k^\eta| +\Big|\frac{\phi_1^\varepsilon*\big(\rho^\eta_m -\rho^\eta\big)k^\eta}{(\nu+ \phi_1^\varepsilon*\rho^\eta_m)(\nu+ \phi_1^\varepsilon*\rho^\eta)}\Big|.
\end{align}
For the first term in the right hand side of (\ref{up}), using (\ref{stc1}) gives that 
\begin{align}\label{k}
\|k^\eta_m-k^\eta\|_{L^q([0,T]\times\mathbb{R}^d)}
&=\Big\|\int_{\mathbb{R}^d}v\phi_2^\delta(v)\cdot\phi_1^\varepsilon*( f^\eta_m -f^\eta) dv\Big\|_{L^q([0,T]\times\mathbb{R}^d)}\nonumber\\
&\le \Big\|\int_{\mathbb{R}^d}v\phi_2^\delta(v)(f^\eta_m -f^\eta)dv\Big\|_{L^q([0,T]\times\mathbb{R}^d)}\to 0\quad \mbox{as}\ m\to \infty.
\end{align} 
Then, from (\ref{k1}), we have
$$\Big|\frac{\big(\phi_1^\varepsilon*\rho^\eta_m -\phi_1^\varepsilon*\rho^\eta\big)k^\eta}{(\nu+ \phi_1^\varepsilon*\rho^\eta_m)(\nu+ \phi_1^\varepsilon*\rho^\eta)}\Big|\le \frac{1}{\nu}|k^\eta|+\frac{1}{\nu}|k^\eta|=\frac{2}{\nu}|k^\eta|\in L^q([0,T]\times\mathbb{R}^d),$$
and using convergence of $\{\rho^\eta_m\}_{m\in \mathbb{N}_+}$, we find
$$\Big|\frac{\phi_1^\varepsilon*\big(\rho^\eta_m -\rho^\eta\big)k^\eta}{(\nu+ \phi_1^\varepsilon*\rho^\eta_m)(\nu+ \phi_1^\varepsilon*\rho^\eta)}\Big|\to 0\qquad \mbox{as}\ m\to\infty\quad \mbox{a.e.}.$$
Thus by the Lebesgue dominated convergence theorem, we have
\begin{align}\label{lastterm}
\lim\limits_{m\to\infty}\int_0^T\int_{\R^d}\Big|\frac{\phi_1^\varepsilon*\big(\rho^\eta_m -\rho^\eta\big)k^\eta}{(\nu+ \phi_1^\varepsilon*\rho^\eta_m)(\nu+ \phi_1^\varepsilon*\rho^\eta)}\Big|^qdxdt = 0.
\end{align}
Substitute (\ref{k}) and (\ref{lastterm}) into (\ref{up}), we can obtain
$$
u^\eta_m\to u^\eta\qquad \mbox{in}\ L^q([0,T]\times\mathbb{R}^d)\quad \mbox{as}\ m\to\infty.
$$
To sum up, we finish the proof of compactness to $\mathcal{T}$.
\\ 

Since the operator $\mathcal{T}$ is well-defined, continuous and compact, we can use the Schauder’s fixed point theorem to obtain weak solutions of equation (\ref{pde3}) and the following entropy inequality in Ref.\cite[Subsection 5.1]{carrillo2021quantifying},
\begin{align}\label{entropy1}
\int_{\R^{2d}}&\Big(\frac{|v|^2}{2}+\frac{|x|^2}{2}+V^R+\sigma |logf^\eta|\Big)f^\eta dxdv 
+\frac{1}{2}\int_{\mathbb{R}^{2d}}W_\varepsilon(x-y)\rho^\eta(x)\rho^\eta(y)dxdy\nonumber\\
&+\int_0^T\int_{\mathbb{R}^{2d}}\frac{1}{f^\eta}\Big|\sigma\nabla_vf^\eta-(v-\chi_\zeta(u^\eta))f^\eta\Big|^2dxdvdt
\nonumber\\
\le& \int_{\mathbb{R}^{2d}}\Big(\frac{|v|^2}{2}+\frac{|x|^2}{2}+V^R+\sigma |logf_0^\eta|\Big)f_0^\eta dxdv+\sigma dT\|f_0\|_{L^1} \nonumber\\
&+\frac{1}{2}\int_{\mathbb{R}^{2d}}W_\varepsilon(x-y)\rho_0^\eta(x)\rho_0^\eta(y)dxdy + \int_0^T\int_{\mathbb{R}^{2d}}(|v|^2+|x|^2)f^\eta dxdvdt +C,
\end{align}
where $C$ is a constant independent of $R, \zeta$.
\\

Step 3: we prove the existence of  weak solutions to (\ref{pde2}). 

It remains to prove the convergence as $R, \zeta\to \infty$. Set $R = \zeta$ and we tend $R$ to infinity.  (Indeed, although $f^\eta$ is only integrable, it can be approached by a function in $\mathcal{C}_c^\infty([0,T]\times \mathbb{R}^{2d})$. Multiplying by $p(f^\eta)^{p-1}, p\ge 1$ on both side of
equation (\ref{pde3}) and integrating on $\mathbb{R}^d\times\mathbb{R}^d$
leads to
\begin{align*}
\frac{d}{dt}&\int_{\R^d}\int_{\R^d} (f^\eta)^pdxdv\\
=&(p-1)\int_{\R^d} \int_{\R^d}(f^\eta)^p\nabla_v\cdot\Big(\gamma v+\lambda(\nabla_x V^R +\nabla_x W_\varepsilon*\rho^\eta)+ \beta(v-\chi_R(u^\eta))\Big)dxdv\\
&-\sigma p(p-1)\int_{\R^d}\int_{\R^d} (f^\eta)^{p-2}|\nabla_v f^\eta|^2dxdv\\
=& d(\gamma+\beta)(p-1)\int_{\R^d}\int_{\R^d} (f^\eta)^p dxdv -\frac{4\sigma (p-1)}{p}\int_{\R^d}\int_{\R^d} |\nabla_v (f^\eta)^{p/2}|^2dxdv.
\end{align*} 
This together with  $\mathrm{Gr\ddot{o}nwall}$’s lemma gives that for $p\in[1, \infty)$, it holds
\begin{align}\label{feta}
\| f^\eta&\|_{L^\infty(0,T; L^p(\R^{2d}))}^p + \frac{4\sigma (p-1)}{p} \int_0^T \|\nabla_v(f^\eta
)^{p/2}(\cdot,\cdot,t)\|_{L^2(\R^{2d})}^2dt\le \| f_0\|_{L^p(\R^{2d})}^pe^{{d(\gamma+\beta)(p-1)}T}.
\end{align}
In particular, we have for any $t\in[0,T]$,
\begin{align}\label{f_inf1}
\|f^\eta(\cdot,\cdot,t)\|_{L^1(\R^{2d})}\le\|f_0\|_{L^1(\R^{2d})},\qquad \|f^\eta(\cdot,\cdot,t)\|_{L^\infty(\R^{2d})}\le e^{Ct}\|f_0\|_{L^\infty(\R^{2d})}.
\end{align}
Moreover, we use $\mathrm{Gr\ddot{o}nwall}$’s lemma to yield, for $t\in[0,T]$, by the entropy inequality (\ref{entropy1}) we can obtain
\begin{align}\label{2m}
\int_{\R^d}\int_{\R^d} \Big(\frac{|v|^2}{2}+\frac{|x|^2}{2}\Big)f^\eta dxdv\le C(T).
\end{align}
From \cite[Lemma 2.4]{karper2013existence}, we know that
\begin{align}
&\parallel \rho^\eta\parallel_{L^\infty(0,T;L^{q_1}(\mathbb{R}^d))}\le C\qquad \forall q_1\in [1,(d+2)/d),\label{rhoeta1}\\
&\parallel j^\eta\parallel_{L^\infty(0,T;L^{q_2}(\mathbb{R}^d))}\le C\qquad \forall q_2\in[1, (d+2)/(d+1)),\label{rhoeta2}
\end{align}
where $C$ is a constant independent of $\eta$.
Thus, by (\ref{feta}), (\ref{rhoeta1}) and (\ref{rhoeta2}), we have that there exist subsequences of $f^\eta$, $\nabla_v f^\eta$, $\rho^\eta$ and $j^\eta$, without relabeling for convenience, we have the following weak convergences as $R\to\infty$
\begin{align}
& f^\eta\rightharpoonup f_\xi\qquad \mbox{in}\ L^\infty(0,T;L^p(\mathbb{R}^{2d}))\quad \forall p\in[1,\infty],\label{f11}\\
& \rho^\eta\rightharpoonup\rho_\xi \qquad  \mbox{in}\ L^\infty(0,T;L^{q_1}(\mathbb{R}^d))\quad \forall q_1\in[1,(d+2)/d),\\
&  j^\eta
\rightharpoonup j_\xi \qquad  \mbox{in}\ L^\infty(0,T; L^{q_2}(\mathbb{R}^d))\quad \forall q_2\in[1,(d+2)/(d+1)),\label{j111}\\
&\qquad \nabla_vf^\eta\rightharpoonup \nabla_v f_\xi \qquad  \mbox{in}\  L^2([0,T]\times\R^{2d}).\label{f111}
\end{align}

Next, we define $G^\eta$ as
\begin{align}\label{G11}
G^\eta:=\Big(\gamma v+\lambda(\nabla_x V^R +\nabla_x W_\varepsilon*\rho^\eta)+ \beta(v-\chi_R(u^\eta))\Big)f^\eta+\sigma\nabla_vf^\eta.
\end{align}
Just like the estimate to (\ref{G}) in Step 2,  the only
additional difficulty is to bound uniformly the term $\chi_R(u^\eta)f^\eta$. For $t\in[0,T]$,
\begin{align*}
\|\chi_R(u^\eta) f^\eta\|_{L^p(\R^{2d})}
&\le\|u^\eta f^\eta\|_{L^p(\R^{2d})}
\le
\|u^\eta\|_{L^p(\R^d)} \|f^\eta\|_{L^\infty(\R^{2d})}\\
&\le\frac{1}{\nu}
\Big\|\int_{\mathbb{R}^{d}}v\phi_2^\delta\cdot\phi_1^\varepsilon*f^\eta dv\Big\|_{L^p(\R^d)} \|f^\eta\|_{L^\infty(\R^{2d})}\\
&\le \frac{1}{\nu}\|\int_{\mathbb{R}^{d}}v\phi_2^\delta\cdot f^\eta dv\Big\|_{L^p(\R^d)}\|f^\eta\|_{L^\infty(\R^{2d})}.
\end{align*}
Here by the definition of $\phi_2^\delta$, we have
\begin{align*}
\int_{\mathbb{R}^{d}}v\phi_2^\delta f^\eta dv 
&\le \int_{\mathbb{R}^{d}}(1+|v|) f^\eta dv 
\le \Big(\int_{\mathbb{R}^{d}}(1+|v|)^2 f^\eta dv \Big)^\frac{1}{p} \Big(\int_{\mathbb{R}^{d}}\frac{f^\eta}{(1+|v|)^{2q/p-q}}dv \Big)^\frac{1}{q}.
\end{align*}
Noticing that $p\in [1,(d+2)/(d+1))$ implies $2q/p-q>d$, we have 
\begin{align*}
\int_{\mathbb{R}^{d}}v\phi_2^\delta f^\eta dv 
\le C\|f^\eta\|_{L^\infty(\R^{2d})}^{1/q}\Big(\int_{\mathbb{R}^{d}}(1+|v|)^2 f^\eta dv \Big)^{1/p},
\end{align*}
which indicates that
\begin{align}\label{bd1}
\Big\|\int_{\mathbb{R}^{d}}v\phi_2^\delta f^\eta dv\Big\|_{L^p(\R^d)} \le  C\int_{\R^d}\int_{\R^d}(1+|v|)^2 f^\eta dvdx. 
\end{align}
Thus, by \eqref{2m} and \eqref{bd1}, we get $$\|\chi_R(u^\eta) f^\eta\|_{L^\infty(0,T;L^p(\R^{2d}))}\le C,$$ 
where $C$ is a constant independent of $R, \varepsilon, \delta$. The boundednesses of others are as in the previous estimates. So we have $G^\eta$ in $L^\infty(0,T;L^p_{loc}(\R^{2d}))$ for all $p\in[1,(d+2)/(d+1))$. Thus, we can set $r = 2$ and apply Lemma \ref{lem2.2} to show that the following strong convergences for $q\in(1,(d+2)/(d+1))$,
\begin{align*}
&\int_{\R^d}f^\eta dv\to \int_{\R^d}f_\xi dv,\quad
\int_{\R^d}v\phi_2^\delta f^\eta dv\to \int_{\R^d}v\phi_2^\delta f_\xi dv\qquad \mbox{in}\ L^q([0,T]\times\mathbb{R}^d)\ \mbox{as}\ R\to\infty,
\end{align*} 
and we consider
\begin{align*}
\|k^\eta-k_\xi\|_{L^q([0,T]\times\R^d)}
&=\Big\|\int_{\mathbb{R}^d}v\phi_2^\delta(v)\cdot\phi_1^\varepsilon(x)*(f^\eta-f_\xi) dv\Big\|_{L^q([0,T]\times\R^d)}\\
&\le \Big\|\int_{\mathbb{R}^d}v\phi_2^\delta(v)(f^\eta-f_\xi) dv\Big\|_{L^q([0,T]\times\R^d)}
\to 0\quad \mbox{as}\ R\to \infty.
\end{align*}
Without relabeling for convenience, we can get the following convergent subsequences for $q\in(1,(d+2)/(d+1))$, 
\begin{align}\label{sc1}
\rho^\eta\to\rho_\xi,\  k^\eta\to k_\xi\qquad \mbox{in}\ L^q([0,T]\times\mathbb{R}^d)\qquad \mbox{as}\ R\to \infty\quad\mbox{a.e.}.
\end{align}

Hence, refer to \cite[(5.7)]{carrillo2021quantifying}, we have
$$(\nabla W_\varepsilon*\rho^\eta)f^\eta\rightharpoonup (\nabla W_\varepsilon*\rho_\xi)f_\xi,
\qquad \qquad  \chi_R(u^\eta) f^\eta \rightharpoonup u_\xi f_\xi$$
as $R\to \infty$ and the processes of proofs used (\ref{sc1}). By combining the weak convergences of $f^\eta$, $vf^\eta$ and $\nabla_v f^\eta$ in (\ref{f11})-(\ref{f111}), we conclude that
$f_\xi$ is a weak solution of (\ref{pde2}) in the following weak sense
\begin{align}\label{ws}
	&\int_0^T\int_{\R^d}\int_{\R^d} -f_\xi\varphi_t - vf_\xi\nabla_x\varphi + \big[\big(\gamma v+\lambda(\nabla_x V +\nabla_x W_\varepsilon*\rho_\xi)\big)f_\xi\big]\nabla_v\varphi dvdxdt \nonumber\\
	&+ \int_0^T\int_{\R^d}\int_{\R^d} \big(\beta(v-u_\xi)f_\xi + \sigma\nabla_v f_\xi\big)\nabla_v\varphi 
	dvdxdt = \int_{\R^d}\int_{\R^d} f_0\varphi(0,\cdot)dvdx
\end{align}
for any $\varphi\in \mathcal{C}_c^\infty([0,T]\times \mathbb{R}^{2d})$,
 and $f_\xi$ satisfies the entropy inequality \eqref{entropy} as Ref.\cite[(5.6)]{carrillo2021quantifying}.
\end{proof}

Now, we provide uniform estimates respect to $\varepsilon,\delta$ for the solutions of the model \eqref{pde2}.
\begin{lem}\label{lem2.3}
Let $f_0$ be in Theorem \ref{thm2.1}, $f_\xi$ is a weak solution of the regularized problem (\ref{pde2}). Then for any $\lambda,\beta,\gamma >0$, there are following uniform estimates
\begin{align*}
&\parallel f_\xi\parallel_{L^\infty(0,T;L^p(\mathbb{R}^{d}\times\mathbb{R}^d))} 
+\parallel \nabla_vf_\xi^{p/2}\parallel_{L^2([0,T]\times\mathbb{R}^{d}\times\mathbb{R}^d)}\\
&\qquad+\parallel \rho_\xi\parallel_{L^\infty(0,T;L^{q_1}(\mathbb{R}^d))} + \parallel j_\xi\parallel_{L^\infty(0,T;L^{q_2}(\mathbb{R}^d))}\le C(T),
\end{align*}
where $p\in[1,\infty]$, $q_1\in[1,(d+2)/d)$ and $q_2\in[1,(d+2)/(d+1))$. In particular, the $L^\infty$- estimate holds
\begin{align*}
\|f_\xi(\cdot,\cdot,t)\|_{L^\infty(\R^{d}\times\R^{d})}\le e^{Ct}\|f_0\|_{L^\infty(\R^{d}\times\R^{d})},
\end{align*}
where $C$ is a positive constant independent of $\varepsilon,\delta$.
\end{lem}
\begin{proof}
Taking $p(f_\xi)^{p-1},p\ge1$ as a test function in (\ref{ws}). (Indeed, although $f_\xi$ is only integrable, it can be approached by a function in $\mathcal{C}_c^\infty([0,T]\times \mathbb{R}^{2d})$.)
\begin{align*}
\frac{d}{dt}&\int_{\R^d}\int_{\R^d} (f_\xi)^pdxdv\\
=&(p-1)\int_{\R^d}\int_{\R^d} (f_\xi)^p\nabla_v\cdot\Big(\gamma v+\lambda(\nabla_x V +\nabla_x W_\varepsilon*\rho_\xi)+ \beta(v-u_\xi)\Big)dxdv\\
&-\sigma p(p-1)\int_{\R^d}\int_{\R^d} (f_\xi)^{p-2}|\nabla_v f_\xi|^2dxdv\\
=& d(\gamma+\beta)(p-1)\int_{\R^d}\int_{\R^d} (f_\xi)^p dxdv -\frac{4\sigma (p-1)}{p}\int_{\R^d}\int_{\R^d}|\nabla_v (f_\xi)^{p/2}|^2dxdv.
\end{align*} 
Hence using $\mathrm{Gr\ddot{o}nwall}$’s lemma, we have that for $p\in[1, \infty)$, it holds
\begin{align*}
\parallel f_\xi&\parallel_{L^\infty(0,T; L^p(\R^{2d}))}^p + \frac{4\sigma (p-1)}{p} \int_0^T \parallel\nabla_v(f_\xi
)^{p/2}(\cdot,\cdot,t)\parallel_{L^2(\R^{2d})}^2dt\le \parallel f_0\parallel_{L^p(\R^{2d})}^pe^{{d(\gamma+\beta)(p-1)}T}.
\end{align*}
Moreover, by the entropy inequality (\ref{entropy}), we can obtain
\begin{align}\label{sm}
\int_{\R^d}\int_{\R^d}\Big(\frac{|v|^2}{2}+\frac{|x|^2}{2}\Big)f_\xi dxdv\le C,
\end{align}
where $C$ is a positive constant independent of $\varepsilon,\delta$.
From \cite[Lemma 2.4]{karper2013existence}, we know that
\begin{align*}
&\parallel \rho_\xi\parallel_{L^\infty(0,T;L^{q_1}(\mathbb{R}^d))}\le C\qquad \forall q_1\in [1,(d+2)/d),\\
&\parallel j_\xi\parallel_{L^\infty(0,T;L^{q_2}(\mathbb{R}^d))}\le C\qquad \forall q_2\in[1, (d+2)/(d+1)).
\end{align*}
Thus, we conclude the proof.

\end{proof}
\subsection{Existence of solutions to the model \eqref{pde1}}
In this subsection, we establish the existence of weak solutions to the problem described in equation \eqref{pde1}. Our proof is structured in two parts: initially, we demonstrate the existence of weak solutions as $\varepsilon$ and $\delta$ approach zero. Subsequently, we examine the convergence of these solutions as $\nu \to 0$. The main findings are detailed in the theorem below.
\begin{thm}\label{thm2.2}
Let $T>0$. Suppose that $f_0$ satisfies
$$ f_0 \in L_+^1\cap L^\infty(\mathbb R^d\times\mathbb{R}^d) \quad and\quad (|v|^2+V+W*\rho_0)f_0 \in L^1(\mathbb R^d\times\mathbb{R}^d).$$
Then for any $\lambda,\beta,\gamma >0$, there exists a weak solution $f$ of Eq.(\ref{pde1}) satisfying
\begin{align*}
&f\in\mathcal{C}([0,T];L^1(\mathbb{R}^d\times\mathbb{R}^d))\cap L^\infty([0,T]\times\mathbb{R}^d\times\mathbb{R}^d),\\
&(|v|^2+V+W*\rho)f \in L^\infty(0,T; L^1(\mathbb R^d\times\mathbb{R}^d)),
\end{align*} 
and the following integral equation
\begin{align}
	&\int_0^T\int_{\mathbb{R}^d}\int_{\mathbb{R}^d} -f\varphi_t - vf\nabla_x\varphi + \big[\big(\gamma v+\lambda(\nabla_x V +\nabla_x W*\rho)\big)f\big]\nabla_v\varphi dvdxdt \nonumber\\
	&+ \int_0^T\int_{\mathbb{R}^d}\int_{\mathbb{R}^d} \big(\beta(v-u)f + \sigma\nabla_v f\big)\nabla_v\varphi 
	dvdxdt = \int_{\mathbb{R}^d}\int_{\mathbb{R}^d} f_0\varphi(0,\cdot)dvdx \label{sol}
    \end{align}
for any $\varphi\in \mathcal{C}_c^\infty([0,T]\times \mathbb{R}^d\times \mathbb{R}^d)$.

\end{thm}
\begin{proof}
Step1: In this step, we prove the solution of \eqref{pde2} converges to the solution of the following equation as $\varepsilon,\delta\to0$
\begin{align}\label{pdev}
\partial_t& f_\nu +v\cdot\nabla_x f_\nu- \nabla_v\cdot[\big(\gamma v+\lambda(\nabla_x V +\nabla_x W*\rho_\nu)\big)f_\nu]= \nabla_v\cdot[\beta(v-u_\nu)f_\nu+\sigma\nabla_vf_\nu],
\end{align}
where $u_\nu$ is defined in the following form
\begin{align*}
u_\nu(x,t) = \frac{\int_{\mathbb{R}^d}v f_\nu dv}{\nu+\int_{\mathbb{R}^d}f_\nu dv}.
\end{align*}
By the uniform estimate in Lemma \ref{lem2.3}, there exist subsequences of $f_\xi$, $\nabla_v f_\xi$ , $\rho_\xi$ and $j_\xi$, without relabeling for convenience, such that the following weak convergences hold as $\varepsilon,\delta\to 0$
\begin{align}
\label{f22}&  f_\xi\rightharpoonup f_\nu \qquad  \mbox{in}\ L^\infty(0,T;L^p(\mathbb{R}^{2d}))\quad \forall p\in[1,\infty],\\
& \rho_\xi\rightharpoonup\rho_\nu \qquad  \mbox{in}\ L^\infty(0,T;L^{q_1}(\mathbb{R}^d))\quad \forall q_1\in[1,(d+2)/d),\\
&  j_\xi\rightharpoonup j_\nu \qquad  \mbox{in}\ L^\infty(0,T; L^{q_2}(\mathbb{R}^d))\quad \forall q_2\in[1,(d+2)/(d+1)),\label{j222}\\
&\qquad\ \nabla_vf_\xi\rightharpoonup \nabla_v f_\nu \qquad 
 \mbox{in}\ L^2([0,T]\times\mathbb{R}^{2d}).\label{f222}
\end{align}
Define $G_\xi$ as
$$G_\xi:=\Big(\gamma v+\lambda(\nabla_x V +\nabla_x W_\varepsilon*\rho_\xi)+ \beta(v-u_\xi)\Big)f_\xi+\sigma\nabla_vf_\xi.$$ 
Just like the estimate to (\ref{G}) and \eqref{G11} in Theorem \ref{thm2.1}, we see that $G_\xi\in L^\infty(0,T;L^p_{loc}{(\mathbb{R}^{2d})})$, $p\in[1,(d+2)/(d+1))$.  The additional difficulty is to bound the term  
$\nabla_x W_\varepsilon*\rho_\xi \cdot f_\xi$
uniformly bounded independent of $\varepsilon, \delta$. Thus, we consider for any $t\in[0,T]$,
\begin{align*}
\|\nabla_x W_\varepsilon*\rho_\xi \cdot f_\xi\|_{L^p(\mathbb{R}^{2d})}
&\le 
\|\nabla_x W_\varepsilon*\rho_\xi\|_{L^p(\mathbb{R}^{d})}
\|f_\xi\|_{L^\infty(\mathbb{R}^{2d})}\\
&\le
\| \nabla_x W\|_{L_w^{d/(d-1)}(\R^{d})}\|\rho_\xi\|_{L^{q_1}(\R^{d})}\|f_\xi\|_{L^\infty(\R^{2d})}\le C,
\end{align*}
where $p\in[1,d/(d-1)]$ and $\nabla_x W=C(d)|x|^{1-d}$.
The boundednesses of others are as in the previous estimate. Thus, we can obtain $G_\xi\in L^\infty(0,T;L^p_{loc}{(\mathbb{R}^{2d})})$ for all $p\in[1,(d+2)/(d+1))$ and apply Lemma \ref{lem2.2} to show that the following strong convergences  similarly as before, for $q\in(1,(d+2)/(d+1))$,
\begin{align}\label{stc2}
&\int_{\R^d}f_\xi dv\to \int_{\R^d}f_\nu dv,
\int_{\R^d}v\phi_2^\delta(v)f_\xi dv\to \int_{\R^d}vf_\nu dv\quad \mbox{in}\ L^q([0,T]\times\mathbb{R}^d)\quad \mbox{as}\ \varepsilon,\delta \to 0.
\end{align} 
and since the weak convergence \eqref{j222}, we have
\begin{align*}
\|k_\xi-j_\nu\|_{L^q([0,T]\times\mathbb{R}^d)}
=&\Big\|\int_{\mathbb{R}^d}v\phi_2^\delta(v)\cdot\phi_1^\varepsilon*f_\xi dv-\int_{\mathbb{R}^d}vf_\nu dv\Big\|_{L^q([0,T]\times\mathbb{R}^d)}\\
=&\Big\|\int_{\mathbb{R}^d}v\phi_2^\delta(v)\cdot\phi_1^\varepsilon*f_\xi dv-\int_{\R^d}v\phi_2^\delta(v)f_\xi dv\Big\|_{L^q([0,T]\times\mathbb{R}^d)}\\
&+\Big\|\int_{\R^d}v\phi_2^\delta(v)f_\xi dv-\int_{\mathbb{R}^d}vf_\nu dv\Big\|_{L^q([0,T]\times\mathbb{R}^d)}
\end{align*}
According to the properties of mollifier,
\begin{align*}
\Big\|\phi_1^\varepsilon*\int_{\mathbb{R}^d}v\phi_2^\delta(v)\cdot f_\xi dv-\int_{\R^d}v\phi_2^\delta(v)f_\xi dv\Big\|_{L^q([0,T]\times\mathbb{R}^d)} \to 0\qquad \mbox{as}\ \varepsilon\to 0, 
\end{align*}
and by \eqref{stc2}, we have
\begin{align*}
\Big\|\int_{\R^d}v\phi_2^\delta(v)f_\xi dv-\int_{\mathbb{R}^d}vf_\nu dv\Big\|_{L^q([0,T]\times\mathbb{R}^d)}\to 0 \qquad \mbox{as} \ \varepsilon,\delta\to 0.    
\end{align*}
Thus, we can get the following convergent subsequences for $q\in(1,(d+2)/(d+1))$,
\begin{align}\label{sc2}
\rho_\xi\to\rho_\nu,\ k_\xi\to j_\nu\qquad \mbox{in}\ L^q([0,T]\times\mathbb{R}^d)\qquad \mbox{as}\ \varepsilon,\delta\to 0\quad \mbox{a.e.}.
\end{align}
Hence, refer to \cite[(5.7),(5.10)]{carrillo2021quantifying}, we also have
$$\nabla W_\varepsilon*\rho_\xi\rightharpoonup\nabla W*\rho_\nu,
\quad  u_\xi f_\xi\rightharpoonup u_\nu f_\nu \qquad\qquad \mbox{as}\ \varepsilon,\delta\to 0.$$
The processes of proofs used (\ref{sc2}). By combining the weak convergences of $f_\xi$, $vf_\xi$ and $\nabla_v f_\xi$ in (\ref{f22})-(\ref{f222}), we conclude that
$f_\nu$ is a weak solution of (\ref{pdev}).
And as Lemma \ref{lem2.3}, we have the following uniform estimates
\begin{align}\label{fnub}
&\parallel f_\nu\parallel_{L^\infty(0,T;L^p(\mathbb{R}^{2d}))} 
+\parallel \nabla_v(f_\nu
)^{p/2}\parallel_{L^2([0,T]\times\mathbb{R}^{2d})}\nonumber\\
&\qquad+\parallel \rho_\nu\parallel_{L^\infty(0,T;L^{q_1}(\mathbb{R}^d))} + \parallel j_\nu\parallel_{L^\infty(0,T;L^{q_2}(\mathbb{R}^d))}\le C(T),
\end{align}
where $p\in[1,\infty)$, $q_1\in[1,(d+2)/d)$ and $q_2\in[1,(d+2)/(d+1))$. In particular, it is also true that
\begin{align*}
\|f_\nu(\cdot,\cdot,t)\|_{L^\infty(\R^{2d})}\le e^{Ct}\|f_0\|_{L^\infty(\R^{2d})},
\end{align*}
where $C$ is a positive constant independent of $\nu$.

Step2: In this step, we prove existence of weak solution to (\ref{pde1}).
By the uniform estimate \eqref{fnub}, there exist subsequences of $f_\nu$, $\nabla_v f_\nu$ , $\rho_\nu$ and $j_\nu$, without relabeling for convenience, we have the following weak convergences hold as $\nu\to0$
\begin{align}
\label{f33}&  f_\nu\rightharpoonup f \qquad  \mbox{in}\ L^\infty(0,T;L^p(\mathbb{R}^{2d}))\quad \forall p\in[1,\infty],\\
& \rho_\nu\rightharpoonup\rho \qquad  \mbox{in}\ L^\infty(0,T;L^{q_1}(\mathbb{R}^d))\quad \forall q_1\in[1,(d+2)/d),\\
&  j_\nu\rightharpoonup j \qquad  \mbox{in}\ L^\infty(0,T; L^{q_2}(\mathbb{R}^d))\quad \forall q_2\in[1,(d+2)/(d+1)),\\
&\qquad\ \nabla_vf_\nu\rightharpoonup \nabla_v f \qquad 
 \mbox{in}\ L^2([0,T]\times\mathbb{R}^{2d}).\label{f333}
\end{align}
For the existence of weak solutions to (\ref{pde1}), it remains to prove the convergence as $\nu\to0$ to (\ref{pdev}).
we define $G_\nu$ as
$$G_\nu:=\Big(\gamma v+\lambda(\nabla_x V +\nabla_x W*\rho_\nu)+ \beta(v-u_\nu)\Big)f_\nu+\sigma\nabla_vf_\nu.$$ 
We see that $G_\nu\in L^\infty(0,T;L^p_{loc}{(\mathbb{R}^{2d}))}, p\in[1,(d+2)/(d+1))$.  The additional difficulty is to bound the term $u_\nu f_\nu$
uniformly with respect to $\nu$. Thus, we consider for $t\in[0,T]$,
\begin{align}\label{uf}
\|u_\nu f_\nu\|_{L^p(\R^{2d})}\le \|f_\nu\|_{L^{p/(2-p)}(\R^{2d})}^{1/2}\|u_\nu\sqrt{f_\nu}\|_{L^2(\R^{2d})},
\end{align} 
where $p\in(1,(d+2)/(d+1))$ and $p/(2-p)\in(1,(d+2)/(d+1))$. We notice that
\begin{align*}
\int_{\R^d}\int_{\R^d}(u_\nu)^2f_\nu dvdx
&\le\int_{\R^d}\frac{\int_{\R^d}f_\nu  dv \Big(\int_{\R^d}|v|^2f_\nu dv\Big)}
{(\nu+\rho_\nu)^2 }\rho_\nu dx\le
\int_{\R^d}\int_{\R^d} |v|^2f_\nu dvdx<C.
\end{align*}
The boundedness of other terms can be obtained using the same method as in the preceding estimates. Thus, we obtain $G_\nu\in L^\infty(0,T;L^p_{loc}{(\R^{2d})})$ for all $p\in[1,(d+2)/(d+1))$ and apply Lemma \ref{lem2.2} to show that the following strong convergences  similarly as before, for $q\in(1,(d+2)/(d+1))$,
\begin{align}\label{nu}
&\rho_\nu \to\rho,\quad
j_\nu \to j\qquad \mbox{in}\ L^q([0,T]\times\mathbb{R}^d)\quad \mbox{as}\ \nu\to 0.
\end{align} 

To show that $f$ is a weak solution to (\ref{pde1}), it is necessary for us to consider the following convergences in distribution sense, using the weak convergences of $f_\nu$, $vf_\nu$ and $\nabla_v f_\nu$ have been showed in (\ref{f33})-(\ref{f333}), we have
\begin{align}\label{wc}
\nabla W*\rho_\nu\rightharpoonup\nabla W*\rho,
\quad u_\nu f_\nu \rightharpoonup uf \qquad\qquad \mbox{as}\ \nu\to 0.
\end{align}
 The first term has already been proofed in \cite[(5.10)]{carrillo2021quantifying}. For the second term, choose a test function $\psi\in\mathcal{C}_c^\infty([0,T]\times\R^d)$ and $\varphi\in \mathcal{C}_c^\infty(\R^d)$, we
write $\rho_\varphi^\nu:=\int_{\R^d}f_\nu\varphi(v)dv$. Let $\Psi(x,v,t):=\psi(x,t)\varphi(v)$, then
\begin{align*}
\int_0^T\int_{\R^d}\int_{\R^d}u_\nu f_\nu\Psi(x,v,t) dxdvdt =\int_0^T\int_{\mathbb{R}^d}u_\nu\rho_\varphi^\nu\psi(x,t) dxdt.
\end{align*}
Similar to (\ref{uf}), we can obtain the boundedness of $u_\nu\rho_\varphi^\nu$
in $L^\infty(0,T;L^p(\mathbb{R}^d))$, namely,
\begin{align}\label{uf}
\|u_\nu \rho_\varphi^\nu\|_{L^p(\R^d)}
&\le \|\varphi\|_{L^\infty(\R^d)}\|\rho_\nu\|_{L^{p/(2-p)}(\R^d)}^{1/2}\|u_\nu\sqrt{\rho_\nu}\|_{L^2(\R^d)}\nonumber\\
&\le C \int_{\R^d}\frac{\int_{\R^d}f_\nu  dv \Big(\int_{\R^d}|v|^2f_\nu dv\Big)}
{(\nu+\rho_\nu)^2 }\rho_\nu dx\le C \int_{\R^d}\int_{\R^d} |v|^2f_\nu dvdx.
\end{align} 
Thus, there is $M$ such that, up to a subsequence,
\begin{align}\label{ur}
u_\nu\rho_\varphi^\nu\rightharpoonup M \qquad  \mbox{in}\ L^\infty(0,T;L^p(\R^d))\qquad\forall p\in (1,(d+2)/(d+1)).    
\end{align}

Next, we derive what M is. Let $h_1, h_2>0$ and define
\begin{align*}
\mathcal{A}_{h_1}^{h_2}:=\big\{(x,t)\in B(0,h_1)\times(0,T):\rho(x,t)>h_2\big\}.
\end{align*} 
For each $h_1$ and $h_2$, we combine the almost everywhere convergence of $\big(\rho_\nu, j_\nu\big)$ to $(\rho,j)$ in (\ref{nu}) with the Egorov's theorem to deduce that for every $\mu> 0$, choose $\mathcal{A}_\mu\subset \mathcal{A}_{h_1}^{h_2}$ satisfying
$$\big|\mathcal{A}_{h_1}^{h_2}\backslash \mathcal{A}_\mu\big|< \mu\quad \mbox{and}\quad \big(\rho_\nu, j_\nu\big) \to(\rho,j)\qquad \mbox{as}\ \ \nu\to0\ \mbox{uniformly}\ \mbox{on} \ \mathcal{A}_\mu.$$
Then, for a sufficiently small $\nu$, we can obtain $\rho_\nu>h_2/2$ on $\mathcal{A}_\mu$. Consider
\begin{align*}
\int_{\mathcal{A}_\mu}
u_\nu\rho_\varphi^\nu- u\rho_\varphi dxdt
=&\int_{\mathcal{A}_\mu}
(u_\nu- u)\rho_\varphi^\nu dxdt +\int_{\mathcal{A}_\mu}
u(\rho_\varphi^\nu- \rho_\varphi) dxdt\\
=&\int_{\mathcal{A}_\mu}
\Big(\frac{1}{\nu+\rho_\nu}-\frac{1}{\rho}\Big)j_\nu\cdot\rho_\varphi^\nu dxdt+\int_{\mathcal{A}_\mu}
\frac{1}{\rho}(j_\nu-j)\cdot\rho_\varphi^\nu dxdt 
\\
&+\int_{\mathcal{A}_\mu}\int_{\R^d}
u(f_\nu- f)\varphi(v)dvdxdt := K_1 + K_2 +K_3.
\end{align*}

For $K_1$, since $\rho_\nu\to \rho$ a.e. in (\ref{nu}) and 
\begin{align*}
\|\rho_\varphi^\nu\|_{L^\infty([0,T]\times\R^d)}
=&\Big\|\int_{\R^d}f_\nu\varphi(v)dv\Big\|_{L^\infty([0,T]\times\R^d)}  \\
\le&\|f_\nu\|_{L^\infty([0,T]\times\R^{2d})}\int_{\R^d}\varphi(v)dv\le Ce^{Ct}\|f_0\|_{L^\infty(\R^{2d})},
\end{align*}
we have that
\begin{align*}
\Big|\Big(\frac{1}{\nu+\rho_\nu}-\frac{1}{\rho}\Big)j_\nu\cdot\rho_\varphi^\nu\Big|
&\le \frac{C\|\rho_\varphi^\nu\|_{L^\infty(\R^d)}}{h_2}|j_\nu|\le \frac{C}{h_2}|j_\nu|\qquad \mbox{on}\ \mathcal{A}_\mu.
\end{align*}
Thus, we can use the dominated
convergence theorem to get
\begin{align*}
&K_1= \int_{\mathcal{A}_\mu}
\Big(\frac{1}{\nu+\rho_\nu}-\frac{1}{\rho}\Big)j_\nu\cdot\rho_\varphi^\nu dxdt
\to 0\qquad \mbox{as}\ \nu\to 0 .
\end{align*}

For $K_2$, since $\rho>h_2$ in $\mathcal{A}_\mu$,  we estimate from (\ref{nu}),
\begin{align*}
K_2=\int_{\mathcal{A}_\mu}
\frac{1}{\rho}(k_\nu-j)\cdot\rho_\varphi^\nu dxdt\le \frac{1}{h_2}\|j_\nu-j\|_{L^q(\mathcal{A}_\mu)}\|\rho_\varphi^\nu\|_{L^{q'}(\mathcal{A}_\mu)}\to 0\qquad\mbox{as}\ \nu\to 0,
\end{align*}
where $q\in(1,(p+2)/(p+1))$.
For the estimate of $K_3$, since $u\varphi\in L^p(\mathcal{A}_\mu \times\R^{d})$ and $f_\nu\rightharpoonup f$ in $L^\infty(0,T;L^q(\R^{2d}))$ for some $q\in[1,\infty]$,  obviously,
\begin{align*}
K_3=\int_{\mathcal{A}_\mu}\int_{\R^d}
u(f_\nu- f)\varphi dvdxdt\to 0\qquad \mbox{as}\  \nu\to 0.
\end{align*}

In summary,  
$$\int_{\mathcal{A}_\mu}
u_\nu\rho_\varphi^\nu- u\rho_\varphi dxdt\to 0\qquad \mbox{as}\  \nu\to 0,$$
and by \eqref{ur}, we have
$$M=u\rho_\varphi\qquad \mbox{on}\ \ \mathcal{A}_\mu.$$
Since the choices of $h_1, h_2$ and $\mu$ are arbitrary, we now obtain
$$M=u\rho_\varphi\qquad \mbox{on}\ \{\rho>0\}.$$
Furthermore,
\begin{align*}
\int_0^T\int_{\mathbb{R}^{d}}\int_{\mathbb{R}^{d}}
u_\nu  f_\nu \Psi dxdvdt& = \int_0^T\int_{\mathbb{R}^{d}}
u_\nu\rho_\varphi^\nu \psi dxdt\\
&\to \int_0^T\int_{\mathbb{R}^{d}}
u \rho_\varphi\psi dxdt
=\int_0^T\int_{\mathbb{R}^{d}}\int_{\mathbb{R}^{d}}
uf\Psi dxdvdt\quad \mbox{as}\ \nu\to 0.
\end{align*}
Thus, for all test functions $\Psi$, we obtain
\begin{align*}
\lim\limits_{\nu\to0}\int_0^T\int_{\mathbb{R}^{d}}\int_{\mathbb{R}^{d}}
u_\nu f_\nu\Psi dxdvdt= 
\int_0^T\int_{\mathbb{R}^{d}}\int_{\mathbb{R}^{d}}
 uf\Psi dxdvdt,
\end{align*}
which implies $u_\nu f_\nu$ weakly converges to $uf$. Therefore, $f$ is a weak solution to \eqref{pde1}.

\end{proof}

\section{Error Estimation for the Mean-Field Limit}

In this section, we primarily focus on establishing the error estimate for solutions in the expectation sense for the stochastic system with moderately many particles, denoted as (\ref{sde1}), and the mean-field system, denoted as (\ref{sde2}), as outlined in Theorem \ref{thm3.2}. Additionally, we demonstrate the existence and uniqueness of solutions for these systems.

\subsection{Unique existence of solutions to systems \eqref{sde1} and \eqref{sde2}}
Firstly, we establish the existence and uniqueness of solutions for systems (\ref{sde1}) and (\ref{sde2}). Notice that for fixed $\xi >0$, since $\nabla_x W_\varepsilon$ is bounded and Lipschitz continuous in (\ref{W}), we can obtain the following result by standard SDE theory.
\begin{lem}
For any fixed $\xi$, the problem (\ref{sde1}) has a unique global solution $(X^i_{\xi,N}(t), V^i_{\xi,N}(t))$.
\end{lem}
Then, with the help of Theorem \ref{thm2.1} and \ref{thm2.2}, we have the following result.
\begin{thm}\label{thm3.1}
If the regularized problem (\ref{pde2}) has a unique solution $f_\xi \in L^\infty\big(0,T;L^1 \cap L^\infty(\mathbb R^d\times \mathbb R^d)\big)$, then the initial value problem (\ref{sde2}) has a unique global solution $(\overline{X}_\xi(t),\overline{V}_\xi(t),f_\xi)$.
\end{thm}
\begin{proof}
Let $f_\xi\in L^\infty\big(0,T;L^1 \cap L^\infty(\mathbb R^{2d})\big)$ be the unique solution of the regularized problem (\ref{pde2}) with the initial data $f_\xi(x,v,0)=f_0$. Then the following stochastic particles system 
	\begin{align*}
         \left\{\begin{aligned}
		d\widetilde{X}_\xi(t) =&\widetilde{V}_\xi(t) dt,\nonumber\\
		d\widetilde{V}_\xi(t)=& \sqrt{2\sigma}dB(t)-\gamma \widetilde{V}_\xi(t)dt
		-\beta\big(\widetilde{V}_\xi(t)- u_\xi(\widetilde{X}_\xi(t))\big)dt\nonumber\\
		 &- \lambda\Big(\nabla V(\widetilde{X}_\xi(t)) + \nabla W_\varepsilon* \rho_\xi(\widetilde{X}(t))\Big)dt
        ,\end{aligned}\right.
	\end{align*}
where 	\begin{eqnarray*}
\rho_\xi (\widetilde{X}(t)) = 
\int_{\mathbb{R}^d}
f_\xi(\widetilde{X}(t),v) dv,\qquad
u_\xi(\widetilde{X}(t)) = \frac{\int_{\mathbb{R}^d}v\phi_2^\delta(v)\cdot\phi_1^\varepsilon*f_\xi(\widetilde{X}(t),v) dv}{\phi_1^\varepsilon * \rho_\xi(\widetilde{X}(t)) + \nu},
\end{eqnarray*} 
has a unique global solution $(\widetilde{X}(t), \widetilde{V}(t))$ because the coefficients are globally Lipschitz for any fixed $\xi$. Denote by $\widetilde{f}_\xi$ as the probability density function of $(\widetilde{X}(t), \widetilde{V}(t))$, then it follows  from It\^{o}'s formula  that for any smooth test function $\varphi(x,v,t)\in C_c^\infty([0,T]\times \mathbb{R}^{2d})$, it holds
    \begin{align*}
    	&\varphi\big(\widetilde{X}_\xi(T),\widetilde{V}_\xi(T), T\big) - \varphi\big(\widetilde{X}_\xi(0),\widetilde{V}_\xi(0), 0\big) \\
    	=& \int_0^T \Big[-\partial_t\varphi\big(\widetilde{X}_\xi(t),\widetilde{V}_\xi(t), t\big)  
		- \widetilde{V}_\xi(t)\nabla_x \varphi\big(\widetilde{X}_\xi(t),\widetilde{V}_\xi(t), t\big)
		+ \gamma \widetilde{V}_\xi(t)\nabla_v\varphi\big(\widetilde{X}_\xi(t),\widetilde{V}_\xi(t), t\big)\\
 		&+ \Big(\lambda\big(\nabla V(\widetilde{X}_\xi(t)) + \nabla W_\varepsilon* \rho_\xi(\widetilde{X}_\xi(t))\big) 
		+\beta\big(\widetilde{V}_\xi(t)-u_\xi(\widetilde{X}_\xi(t))\big) \Big)\nabla_v \varphi\big(\widetilde{X}_\xi(t),\widetilde{V}_\xi(t), t\big) \\
		&- \sigma\Delta_v\varphi\big(\widetilde{X}_\xi(t),\widetilde{V}_\xi(t), t\big)\Big]dt 
		+ \sqrt{2\sigma}\int_0^T\nabla_v\varphi\big(\widetilde{X}_\xi(t),\widetilde{V}_\xi(t), t\big)dB(t),  
    \end{align*}
by taking the expectation, we get 
    \begin{align*}
	&\int_{\R^d}\int_{\R^d} \widetilde{f}_\xi(x,v,T) \varphi(x,v,T)dxdv - \int_{\R^d}\int_{\R^d} f_0 \varphi(x,v,0)dxdv \\
	=& \int_0^T \int_{\R^d} \int_{\R^d}\widetilde{f}_\xi(x,v,t) \Big[-\partial_t\varphi(x,v,t) 
	-v\nabla_x \varphi(x,v,t)+ \gamma v\nabla_v \varphi(x,v,t)\\
 	&+ \Big(\lambda(\nabla V + \nabla W_\varepsilon* \rho_\xi(x,t) )
	+\beta\big(v -  u_\xi(x,t) \big)\Big)\nabla_v \varphi(x,v,t) - \sigma\Delta_v\varphi(x,v,t) \Big]dxdvdt\,.
    \end{align*}
Thus $\widetilde f_\xi$ satisfies the weak formulation of (\ref{pde2}) also. Thus we have  $\widetilde{f}_\xi=f_\xi$ due to the uniqueness. 
In other words, the unique solution to (\ref{sde2})  is given by $(\widetilde{X}_\xi(t),\widetilde{V}_\xi(t),f_\xi)$, which then can be denoted as $(\overline{X}_\xi(t),\overline{V}_\xi(t),f_\xi)$.
\end{proof}

\subsection{Convergence estimates for $N\to\infty$}
In the following subsection, we establish the error estimate between solutions of the stochastic moderately many-particle system (\ref{sde1}) and the mean-field stochastic system (\ref{sde2}). Prior to this analysis, we offer an estimate for the disparity between their respective local alignment terms.
\begin{lem}\label{lem3.2}
Let $\{(X_{\xi,N}^i(t), V_{\xi,N}^i(t))_{0\le t\le T}\}_{i=1}^N$ and $\{(\overline{X}_{\xi}^i(t), \overline{V}_{\xi}^i(t))_{0\le t\le T}\}_{i=1}^N$ be solutions of equations (\ref{sde1}) and (\ref{sde2}) up to $T$. Then the following estimate for the local alignment terms $u_\xi$ and $\overline{u}_\xi$ holds
\begin{align*}
 \int_{0}^T\mathbb{E}\big[\big|&u_\xi(X_{\xi,N}^i(t)) - \overline{u}_\xi(\overline{X}_{\xi}^i(t))\big|^2\big]dt
\le C\Big(\frac{1}{\delta^2\nu^4\varepsilon^{4d+2}}\int_0^T\mathbb{E}\big[|X_{\xi,N}^i(t) - \overline{X}_{\xi}^i(t)|^2\big] dt \\
&+\frac{1}{\nu^4\varepsilon^{4d}}\int_0^T\mathbb{E}\big[|V_{\xi,N}^i(t) - \overline{V}_{\xi}^i(t)|^2\big] dt + \frac{T}{N\delta^2\nu^4\varepsilon^{4d}}\Big),
\end{align*}
for any $i\in [N]$, where $C$ is a constant only depending on $ \| f_0\|_ {L^1(\mathbb{R}^d\times\mathbb{R}^d)}$.
\end{lem}
\begin{proof}
Let $\forall\ i \in [N]$. From (\ref{u1}) and (\ref{u2}), $u_\xi$ and $\overline{u}_\xi$ can be rewritten as follows
\begin{align*}
u_\xi(X_{\xi,N}^i(t)) 
&=\frac{\frac{1}{N}\sum_{j=1}^N V_{\xi,N}^j(t)\phi_2^\delta(V_{\xi,N}^j(t))\cdot\phi_1^\varepsilon(X_{\xi,N}^i(t) - X_{\xi,N}^j(t) )}{\frac{1}{N}\sum_{j=1}^N \phi_1^\varepsilon(X_{\xi,N}^i(t) - X_{\xi,N}^j(t) ) + \nu},\\ 
\overline{u}_\xi(\overline{X}_{\xi}^i(t))
&= \frac{\int_{\mathbb{R}^d}v\phi_2^\delta(v)\cdot\phi_1^\varepsilon*f_\xi (\overline{X}_{\xi}^i(t))dv}{\phi_1^\varepsilon * \rho_\xi(\overline{X}_{\xi}^i(t)) + \nu}.
\end{align*}
By taking their  difference and applying the expectation lead to
\begin{align}\label{u-u}
&\int_0^T\E\big[\big|u_\xi(X_{\xi,N}^i(t)) - \overline{u}_\xi(\overline{X}_{\xi}^i(t))\big|^2\big]dt\nonumber\\
\le& \frac{1}{\nu^4}\int_0^T
\E\Big[\Big|\frac{1}{N}\sum_{j=1}^N V_{\xi,N}^j(t)\phi_2^\delta(V_{\xi,N}^j(t))\phi_1^\varepsilon(X_{\xi,N}^i(t) - X_{\xi,N}^j(t) )\Big(\phi_1^\varepsilon* \rho_\xi(\overline{X}_{\xi}^i(t)) + \nu\Big)\nonumber\\
&\qquad\qquad-\int_{\R^d}v\phi_2^\delta(v)\phi_1^\varepsilon*f_\xi(\overline{X}_{\xi}^i(t))dv \Big(\frac{1}{N}\sum_{j=1}^N \phi_1^\varepsilon(X_{\xi,N}^i(t) - X_{\xi,N}^j(t) ) + \nu\Big)\Big|^2\Big]dt\nonumber
\\
\le& \frac{2}{\nu^4}\int_0^T
\E\Big[\Big|\Big(\phi_1^\varepsilon* \rho_\xi(\overline{X}_{\xi}^i) + \nu\Big)
\Big(\frac{1}{N}\sum_{j=1}^N V_{\xi,N}^j\phi_2^\delta(V_{\xi,N}^j)\phi_1^\varepsilon(X_{\xi,N}^i - X_{\xi,N}^j )\nonumber\\
&\qquad\qquad\qquad\qquad\qquad\qquad\qquad\qquad\qquad-\int_{\R^d}v\phi_2^\delta(v)\phi_1^\varepsilon*f_\xi(\overline{X}_{\xi}^i)dv \Big)\Big|^2\Big]dt\nonumber\\
&+\frac{2}{\nu^4}\int_0^T
\E\Big[\Big|\int_{\R^d}v\phi_2^\delta(v)\phi_1^\varepsilon*f_\xi(\overline{X}_{\xi}^i)dv \Big(\phi_1^\varepsilon* \rho_\xi(\overline{X}_{\xi}^i) -\frac{1}{N}\sum_{j=1}^N \phi_1^\varepsilon(X_{\xi,N}^i - X_{\xi,N}^j ) \Big|^2\Big]dt\nonumber
\\
:=& \frac{2}{\nu^4}( I_1 +I_2).
\end{align}

The estimate for $I_1$ is
\begin{align*}
I_1 \le &\|\phi_1^\varepsilon*\rho_\xi + \nu\|_{L^\infty(\R^d)}^2\int_0^T\frac{1}{N^2}
\E\Big[\Big|\sum_{j=1}^N \Big(V_{\xi,N}^j(t)\phi_2^\delta(V_{\varepsilon,N}^j(t))\phi_1^\varepsilon\big(X_{\xi,N}^i(t) - X_{\xi,N}^j(t)\big) \nonumber\\
&\qquad\qquad\qquad\qquad\qquad- \int_{\R^d}v\phi_2^\delta(v)\phi_1^\varepsilon*f_\xi(\overline{X}_{\xi}^i(t)) dv\Big) \Big|^2\Big]dt.
\end{align*}
Here by Young's convolution inequation, we have
\begin{align*}
\|\phi_1^\varepsilon*\rho_\xi+\nu\|_{L^\infty(\R^{d})}^2 
&\le 2\|\phi_1^\varepsilon* \rho_\xi\|_{L^\infty(\R^{d})}^2 + 2\nu^2
\le \frac{2}{\varepsilon^{2d}}\| \rho_\xi\|_{L^1(\R^{d})}^2+2\nu^2,
\end{align*}
and 
\begin{align*}
&\int_{0}^T\frac{1}{N^2}
\mathbb{E}\Big[\Big|\sum_{j=1}^N \Big(V_{\xi,N}^j\phi_2^\delta(V_{\xi,N}^j)\phi_1^\varepsilon(X_{\xi,N}^i - X_{\xi,N}^j )- \int_{\mathbb{R}^d}v\phi_2^\delta(v)\cdot\phi_1^\varepsilon*f_\xi\big(\overline{X}_{\xi}^i\big) dv\Big) \Big|^2\Big]dt\\
\le&\frac{3}{N^2}\int_0^T
\mathbb{E}\Big[\Big|\sum_{j=1}^N \Big(V_{\xi,N}^j\phi_2^\delta(V_{\xi,N}^j)\phi_1^\varepsilon(X_{\xi,N}^i - X_{\xi,N}^j ) -  \overline{V}_\xi^j\phi_2^\delta(\overline{V}_\xi^j)\phi_1^\varepsilon(X_{\xi,N}^i-X_{\xi,N}^j)\Big)\Big|^2\Big]dt\\
&+\frac{3}{N^2}\int_0^T\mathbb{E}\Big[\Big| \sum_{j=1}^{N} \Big(\overline{V}_\xi^j\phi_2^\delta(\overline{V}_\xi^j)\phi_1^\varepsilon(X_{\xi,N}^i-X_{\xi,N}^j)- \overline{V}_\xi^j\phi_2^\delta(\overline{V}_\xi^j)\phi_1^\varepsilon(\overline{X}_{\xi}^i-\overline{X}_\xi^j)\Big) \Big|^2\Big]dt\\
&+\frac{3}{N^2}\int_0^T\mathbb{E}\Big[\Big|\sum_{j=1}^{N} \Big(\overline{V}_\xi^j\phi_2^\delta(\overline{V}_\xi^j)\phi_1^\varepsilon(\overline{X}_{\xi}^i-\overline{X}_\xi^j) - \int_{\mathbb{R}^d}v\phi_2^\delta(v)\cdot\phi_1^\varepsilon*f_\xi\big(\overline{X}_{\xi}^i\big) dv\Big) \Big|^2\Big]dt\\
:=&  3(I_1^1 + I_1^2 +I_1^3).
\end{align*}
Now, we derive the estimates for $I_1^1,\ I_1^2$ and $I_1^3$ separately. For the $I_1^1$,
\begin{align*}
I_1^1 =& \frac{1}{N^2}\int_0^T
\mathbb{E}\Big[\Big|\sum_{j=1}^N \phi_1^\varepsilon(X_{\xi,N}^i-X_{\xi,N}^j) \Big(V_{\xi,N}^j\phi_2^\delta(V_{\xi,N}^j) -  \overline{V}_\xi^j\phi_2^\delta(\overline{V}_\xi^j)\Big)\Big|^2\Big]dt\\
\le& \frac{1}{N^2}\|\phi_1^\varepsilon\|_{L^\infty(\R^d)}^2 \int_0^T N^2
\mathbb{E}\Big[\Big| \Big(V_{\xi,N}^i\phi_2^\delta(V_{\xi,N}^i) -  \overline{V}_{\xi}^i\phi_2^\delta(\overline{V}_{\xi}^i)\Big)\Big|^2\Big]dt\\
\le& \frac{1}{\varepsilon^{2d}} \|
\nabla (\cdot\phi_2^\delta(\cdot))\|_{L^\infty(\R^d)}^2
\int_0^T\mathbb{E} \big[\big|\overline{V}_{\xi}^i - V_{\xi,N}^i\big|^2\big]dt
\le \frac{C}{\varepsilon^{2d}}\int_0^T\mathbb{E} \big[\big|\overline{V}_{\xi}^i - V_{\xi,N}^i\big|^2\big]dt.
\end{align*}
The $I_1^2$ can be handled similarly,
\begin{align*}
I_1^2 &=\frac{1}{N^2} \int_0^T \mathbb{E}\Big[\Big| \sum_{j=1}^{N} \overline{V}_\xi^j\phi_2^\delta(\overline{V}_\xi^j)\Big(\phi_1^\varepsilon(X_{\xi,N}^i-X_{\xi,N}^j) - \phi_1^\varepsilon(\overline{X}_{\xi}^i-\overline{X}_\xi^j) \Big)\Big|^2\Big]dt\\
&\le \frac{1}{N^2}\| \cdot\phi_2^\delta(\cdot)\|_{L^\infty(\mathbb{R}^{d})}^2 \int_0^T\mathbb{E} \Big[\Big|\sum_{j=1}^{N} \Big(\phi_1^\varepsilon(X_{\xi,N}^i-X_{\xi,N}^j) - \phi_1^\varepsilon(\overline{X}_{\xi}^i-\overline{X}_\xi^j) \Big)\Big|^2\Big]dt\\
&\le \frac{4}{N^2\delta^2}\|\nabla\phi_1^\varepsilon\|_{L^\infty(\mathbb{R}^{d})}^2 \int_0^T\mathbb{E} \big[N^2\big|\overline{X}_{\xi}^i - X_{\xi,N}^i\big|^2\big]dt\\
&\le \frac{4}{\delta^2\varepsilon^{2d+2}} \int_0^T\mathbb{E} \big[\big|\overline{X}_{\xi}^i - X_{\xi,N}^i\big|^2\big]dt.
\end{align*}
The $I_1^3$ is estimated as follows
\begin{align*}
I_1^3
=& \frac{1}{N^2}\int_0^T\mathbb{E}\Big[\Big|\sum_{j=1}^{N} \Big(\overline{V}_\xi^j\phi_2^\delta(\overline{V}_\xi^j)\phi_1^\varepsilon(\overline{X}_{\xi}^i-\overline{X}_\xi^j)- \int_{\mathbb{R}^d}v\phi_2^\delta(v)\cdot\phi_1^\varepsilon*f_\xi\big(\overline{X}_{\xi}^i\big) dv\Big) \Big|^2\Big]dt\\
\le& \frac{1}{N^2}
\int_0^T \mathbb{E}\Big[\sum_{j=1}^{N} \Big(\overline{V}_\xi^j\phi_2^\delta(\overline{V}_\xi^j)\phi_1^\varepsilon(\overline{X}_{\xi}^i-\overline{X}_\xi^j)- \int_{\mathbb{R}^d}v\phi_2^\delta(v)\cdot\phi_1^\varepsilon*f_\xi\big(\overline{X}_{\xi}^i\big) dv\Big) \nonumber\\
&\qquad\qquad\quad\sum_{l=1}^{N} \Big(\overline{V}_\xi^l\phi_2^\delta(\overline{V}_\xi^l)\phi_1^\varepsilon(\overline{X}_{\xi}^i-\overline{X}_\xi^l)- \int_{\mathbb{R}^d}v\phi_2^\delta(v)\cdot\phi_1^\varepsilon*f_\xi\big(\overline{X}_{\xi}^i\big) dv\Big) \Big] dt\\ 
\le\ &\frac{1}{N^2}
\sum_{j=1}^{N}\sum_{l=1}^{N} \int_0^T \mathbb{E}\Big[\Big(\overline{V}_\xi^j\phi_2^\delta(\overline{V}_\xi^j)\phi_1^\varepsilon(\overline{X}_{\xi}^i-\overline{X}_\xi^j(t))- \int_{\mathbb{R}^d}v\phi_2^\delta(v)\cdot\phi_1^\varepsilon*f_\xi\big(\overline{X}_{\xi}^i\big) dv\Big)\nonumber\\
&\qquad\qquad\qquad\Big(\overline{V}_\xi^l\phi_2^\delta(\overline{V}_\xi^l)\phi_1^\varepsilon(\overline{X}_{\xi}^i-\overline{X}_\xi^l)- \int_{\mathbb{R}^d}v\phi_2^\delta(v)\cdot\phi_1^\varepsilon*f_\xi\big(\overline{X}_{\xi}^i\big) dv\Big) \Big] dt,
\end{align*}
where for $j\ne l$ the expectation is zero. Hence,
\begin{align*}
|I_1^{3}| &\le\frac{1}{N^2}\sum_{j=1}^{N}\int_0^T \mathbb{E}\Big[\Big|\overline{V}_\xi^j\phi_2^\delta(\overline{V}_\xi^j)\phi_1^\varepsilon(\overline{X}_{\xi}^i-\overline{X}_\xi^j)- \int_{\mathbb{R}^d}v\phi_2^\delta(v)\cdot\phi_1^\varepsilon*f_\xi\big(\overline{X}_{\xi}^i\big) dv\Big|^2\Big]dt \\
&\le\frac{1}{N^2}\| \cdot\phi_2^\delta(\cdot)\|_{L^\infty(\R^d)}^2\sum_{j=1}^{N}\int_0^T \mathbb{E}\Big[\Big|\phi_1^\varepsilon(\overline{X}_{\xi}^i-\overline{X}_\xi^j)+\phi_1^\varepsilon*\rho_\xi\big(\overline{X}_{\xi}^i\big)\Big|^2\Big]dt\le \frac{4T}{N\delta^2\varepsilon^{2d}}\| f_\xi\|_{L^1(\R^{2d})}^2,
\end{align*}
while exploiting the fact that 
$$\| \phi_1^\varepsilon*\rho_\xi
\|_{L^\infty(\R^{d})} \leq \|\phi_1^\varepsilon\|_{L^\infty(\R^{d})}\| \rho_\xi\|_{L^1(\R^{d})} \leq \frac{1}{\varepsilon^d}\| \rho_\xi\|_{L^1(\R^{d})}.$$
Thus,
\begin{align}\label{I1}
I_1  \le3(\frac{2}{\varepsilon^{2d}}\| \rho_\xi\|_{L^1(\R^{d})}^2+2\nu^2)\Big(\frac{1}{\delta^2\varepsilon^{2d+2}}&\int_0^T\mathbb{E}\big[|X_{\xi,N}^i(t) - \overline{X}_{\xi}^i(t)|^2\big] dt \nonumber\\
+&\frac{1}{\varepsilon^{2d}}\int_0^T\mathbb{E}\big[|V_{\xi,N}^i(t) - \overline{V}_{\xi}^i(t)|^2\big] dt + \frac{4T\| f_\xi\|_{L^1(\R^{2d})}^2}{N\delta^2\varepsilon^{2d}}\Big).
\end{align}

On the other hand, the estimate for $I_2$ is
\begin{align*}
I_2 \le & 
\Big\|\int_{\R^d} v\phi_2^\delta(v)\phi_1^\varepsilon*f_\xi dv\Big\|_{L^\infty(\R^d)}^2\int_0^T\frac{1}{N^2}\E\Big[\Big|\sum_{j=1}^N\Big(\phi_1^\varepsilon * \rho_\xi(\overline{X}_{\xi}^i(t)) -  \phi_1^\varepsilon(X_{\xi,N}^i(t) - X_{\xi,N}^j(t) )\Big) \Big|^2\Big]dt.
\end{align*}
Here by the Young's convolution inequation, we have
\begin{align*}
\Big\|\int_{\R^d} v\phi_2^\delta(v)\cdot\phi_1^\varepsilon*f_\xi dv \Big\|_{L^\infty(\R^d)} 
&\le \| \cdot\phi_2^\delta(\cdot)\|_{L^\infty(\R^d)} \|\phi_1^\varepsilon* \rho_\xi  \|_{L^\infty(\R^{d})}\\
&\le \frac{2}{\delta} \|\phi_1^\varepsilon\|_{L^\infty(\R^d)}\| \rho_\xi\|_{L^1(\R^{d})}\le \frac{2}{\delta\varepsilon^{d}} \| \rho_\xi\|_{L^1(\R^{d})},
\end{align*}
and
\begin{align*}
&\int_0^T\frac{1}{N^2}\E\Big[\Big|\sum_{j=1}^N\Big(\phi_1^\varepsilon * \rho_\xi(\overline{X}_{\xi}^i(t)) -  \phi_1^\varepsilon(X_{\xi,N}^i(t) - X_{\xi,N}^j(t) )\Big) \Big|^2\Big]dt\\
\le&\frac{2}{N^2}\int_{0}^T
\mathbb{E}\Big[\Big|\sum_{j=1}^N\Big(\phi_1^\varepsilon*\rho_\xi(\overline{X}_{\xi}^i(t))- \phi_1^\varepsilon(\overline{X}_{\xi}^i(t)-\overline{X}_\xi^j(t))\Big) \Big|^2\Big]\\ 
&\qquad+\mathbb{E}\Big[\Big|\sum_{j=1}^N\Big( \phi_1^\varepsilon(\overline{X}_{\xi}^i(t)-\overline{X}_\xi^j(t))- \phi_1^\varepsilon(X_{\xi,N}^i(t)-X_{\xi,N}^j(t))\Big)\Big|^2\Big]dt\\
:=&\ 2(I_2^1+I_2^2).
\end{align*}
The first term is estimated as follows
	\begin{align*}
	I_2^1
		\le\ & \frac{1}{N^2}\int_0^T \mathbb{E}\Big[\sum_{j=1}^N\Big(\phi_1^\varepsilon*\rho_\xi(\overline{X}_{\xi}^i(t))-  \phi_1^\varepsilon(\overline{X}_{\xi}^i(t)-\overline{X}_\xi^j(t))\Big)\nonumber\\
		&\qquad\qquad\qquad\sum_{l=1}^{N}\Big(\phi_1^\varepsilon*\rho_\xi(\overline{X}_{\xi}^i(t))-  \phi_1^\varepsilon(\overline{X}_{\xi}^i(t)-\overline{X}_\xi^l(t))\Big)\Big] dt\\
		=&\ \frac{1}{N^2}\sum_{j=1}^{N}\sum_{l=1}^{N} \int_0^T \mathbb{E}\Big[\Big(\phi_1^\varepsilon*\rho_\xi(\overline{X}_{\xi}^i(t))-  \phi_1^\varepsilon(\overline{X}_{\xi}^i(t)-\overline{X}_\xi^j(t))\Big)\nonumber\\
		&\qquad\qquad\qquad\Big(\phi_1^\varepsilon*\rho_\xi(\overline{X}_{\xi}^i(t))-  \phi_1^\varepsilon(\overline{X}_{\xi}^i(t)-\overline{X}_\xi^l(t))\Big)\Big] dt,
	\end{align*}
where for $j\ne l$ the expectation is zero. Hence,
$$|I_2^1| \le \frac{1}{N^2}\sum_{j=1}^{N} \int_0^T \mathbb{E}\Big[\Big|\phi_1^\varepsilon*\rho_\xi(\overline{X}_{\xi}^i(t))-  \phi_1^\varepsilon(\overline{X}_{\xi}^i(t)-\overline{X}_\xi^j(t))\Big|^2\Big]dt \le \frac{T}{N\varepsilon^{2d}}\| f_\xi\|_{L^1(\R^{2d})}^2. $$
The second term is estimated as follow
\begin{align*}
I_2^2 
\le&\frac{2}{N^2}\|\nabla\phi_1^\varepsilon\|_{L^\infty(\R^{d})}^2\int_0^TN^2\mathbb{E}\big[\big|(X_{\varepsilon,N}^{\delta,i}(t) - \overline{X}_{\xi}^i(t))\big|^2\big] dt\le
\frac{2}{\varepsilon^{2d+2}}\int_0^T\mathbb{E}\big[\big|X_{\xi,N}^i(t) - \overline{X}_{\xi}^i(t)\big|^2\big] dt.
\end{align*}
Thus,
\begin{align}\label{I2}
I_2\le C\frac{4}{\delta^2\varepsilon^{2d}} \| \rho_\xi\|_{L^1(\R^{d})}\Big(\frac{T}{N\varepsilon^{2d}}\| f_\xi\|_{L^1(\R^{2d})}^2+\frac{1}{\varepsilon^{2d+2}}\int_0^T\mathbb{E}\big[|X_{\xi,N}^i(t) - \overline{X}_{\xi}^i(t)|^2\big] dt\Big).
\end{align}

Finally, bring \eqref{I1} and \eqref{I2} into the inequality (\ref{u-u}), we have
\begin{align*}
 \int_{0}^T\mathbb{E}\big[\big|&u_\xi(X_{\xi,N}^i(t)) - \overline{u}_\xi(\bar{X}_\xi^i(t))\big|^2\big]dt
\le C\Big(\frac{1}{\delta^2\nu^4\varepsilon^{4d+2}}\int_0^T\mathbb{E}\big[|X_{\xi,N}^i(t) - \overline{X}_{\xi}^i(t)|^2\big] dt \\
&+\frac{1}{\nu^4\varepsilon^{4d}}\int_0^T\mathbb{E}\big[|V_{\xi,N}^i(t) - \overline{V}_{\xi}^i(t)|^2\big] dt + \frac{T}{N\delta^2\nu^4\varepsilon^{4d}}\Big),
\end{align*}
where $C$ is a constant only depending on $ \| f_0\|_ {L^1}$.

\end{proof}

\subsection{Error estimations of the particles solutions}In this section, by Lemma \ref{lem3.2} and discussion of Section 2, we are able to state the main theorem as following.
\begin{thm}\label{thm3.2}
Let $\{(X_{\xi,N}^i(t), V_{\xi,N}^i(t))_{0\le t\le T}\}_{i=1}^N$ and $\{(\overline{X}_{\xi}^i(t), \overline{V}_{\xi}^i(t))_{0\le t\le T}\}_{i=1}^N$ be solutions of equations (\ref{sde1}) and (\ref{sde2}) up to some $T>0$. Then for any $\gamma, \lambda, \beta >0$ and fixed $0<\alpha\ll1$, $1/(\delta^2\nu^4\varepsilon^{4d+2})\sim \ln( N^\alpha)$, it holds that
\begin{align*}
    \sup\limits_{0<t<T}\sup \limits_{i=1\dots N}  \mathbb{E}\left[|X_{\xi,N}^i(t) - \overline{X}_{\xi}^i(t)|^2+|V_{\xi,N}^i(t) - \overline{V}_{\xi}^i(t)|^2\right] &\leq
\frac{C T^3\ln N^\alpha}{N}\big(1+ T^3 N^{T^3\alpha }\ln N^{\alpha}\big)\,,
    \end{align*}
where $C$ is a constant depending on $\gamma$, $\lambda$, $\beta$ and $ \| f_0\|_ {L^1(\mathbb{R}^d\times\mathbb{R}^d)}$.
\end{thm}
\begin{proof}
Let $\forall\ i \in [N]$. 
By taking the difference of two problems (\ref{sde1}) and (\ref{sde2}), we obtain
    \begin{align*}	
\mathcal{M}^i(t) =|X_{\xi,N}^i(t) - \overline{X}_{\xi}^i(t)|^2
	\le& t\int_0^t  |V_{\xi,N}^i(s) - \overline{V}_{\xi}^i(s)|^2ds, \\
\mathcal{N}^i(t) = |V_{\xi,N}^i(t) - \overline{V}_{\xi}^i(t)|^2
	\le& t\int_0^t  \Big|(\beta+\gamma)\big( \overline{V}_{\xi}^i(s)-V_{\xi,N}^i(s) \big)+\lambda\nabla V\big(\overline{X}_{\xi}^i(s)-X_{\xi,N}^i(s)\big)\\
   	&+ \frac{\lambda}{N} \sum_{j=1}^N\Big(\nabla W_\varepsilon* \rho_\xi\big(\overline{X}_{\xi}^i(s)\big) - \nabla W_\varepsilon \big (X_{\xi,N}^i(s) - X_{\xi,N}^j(s)\big)\Big) \\
   	&+ \beta\Big( u_\xi\big(X_{\xi,N}^i(s)\big)-\overline{u}_\xi\big(\overline{X}_{\xi}^i(s)\big)\Big)\Big|^2 ds.  
    \end{align*}
    
Applying the expectation leads to
\begin{align*}
\mathbb{E}[\mathcal{M}^i(t)] 
&\le t\int_0^t  \mathbb{E}[|V_{\xi,N}^i(s) - \overline{V}_{\xi}^i(s)|^2]ds \le t\int_0^t \mathbb{E}[\mathcal{N}^i(s)]ds,
\end{align*}
and
\begin{align}\label{N}
\mathbb{E}[\mathcal{N}^i(t)]  \leq& \ 4(\beta+\gamma)^2t \int_0^t \mathbb{E}[|\overline{V}_{\xi}^i(s) - V_{\xi,N}^i(s)|^2]ds
 + 4\beta^2t \int_0^t \mathbb{E}[ | u_\xi(X_{\xi,N}^i(s))-\overline{u}_\xi(\overline{X}_{\xi}^i(s)) |^2]ds\nonumber\\
	&+4\lambda^2 t \int_0^t \mathbb{E}\big[\big|\nabla V\big(  \overline{X}_{\xi}^i(s)\big)-\nabla V\big(X_{\xi,N}^i (s)\big)\big|^2\big]ds\nonumber\\
	&+ 4\frac{\lambda^2}{N^2}  t\int_0^t \mathbb{E}\big[ \Big|\sum_{j=1}^{N}\Big(\nabla W_\varepsilon*\rho(\overline{X}_{\xi}^i(s),s ) - \nabla W_\varepsilon\big(X_{\xi,N}^i(s) - X_{\xi,N}^j(s)\big)\Big)\Big|^2\big]ds\nonumber\\
:=&\ \ 4(\beta+\gamma)^2t\int_0^t \mathbb{E}[\mathcal{N}^i(s)]ds + 4\beta^2tJ_1 +4\lambda^2t  J_2+4\lambda^2t  J_3.
	\end{align}  

By Lemma \ref{lem3.2}, $J_1$ can be estimated 
\begin{align}\label{J1}
J_1
\le &\ C \Big(\frac{1}{\delta^2\nu^4\varepsilon^{4d+2}}\int_0^t\mathbb{E}[\mathcal{M}^i(s)] ds +\frac{1}{\nu^4\varepsilon^{4d}}\int_0^t\mathbb{E}[\mathcal{N}^i(s)] ds +\frac{t}{N\delta^2\nu^4\varepsilon^{4d}}\Big),
\end{align}
where $C$ is a constant only depending on $ \| f_0\|_ {L^1(\mathbb{R}^{2d})}$. For $J_2$, simply,
\begin{align}\label{J2}
J_2\le& \|\nabla(\nabla V)\|_{L^\infty(\R^d)}^2\int_0^t \mathbb{E}\big[\big| \overline{X}_{\xi}^i(s)-X_{\xi,N}^i(s)\big|^2\big]ds\le C\int_0^t \mathbb{E}[\mathcal{M}^i(s)]ds,
\end{align}

The $J_3$ can be handled
    \begin{align*}
J_3\le& \frac{3}{N^2} \int_0^t \mathbb{E}\Big[\Big|\sum_{j=1}^N \Big(\nabla W_\varepsilon\big(X_{\xi,N}^i(s)-X_{\xi,N}^j(s)\big) - \nabla W_\varepsilon\big(X_{\xi,N}^i(s)-\overline{X}_\xi^j(s)\big)\Big)\Big|^2\Big] \\ 
        &\qquad+ \mathbb{E}\Big[\Big|\sum_{j=1}^N \Big(\nabla W_\varepsilon\big(X_{\xi,N}^i(s)-\overline{X}_\xi^j(s)\big) - \nabla W_\varepsilon\big(\overline{X}_{\xi}^i(s)-\overline{X}_\xi^j(s)\big)\Big)\Big|^2\Big]\\
        &\qquad+ \mathbb{E}\Big[\Big|\sum_{j=1}^N \Big(\nabla W_\varepsilon(\overline{X}_{\xi}^i(s)-\overline{X}_\xi^j(s)) - \nabla W_\varepsilon*\rho(\overline{X}_{\xi}^i(s),s)\Big)\Big|^2\Big]ds \\
		:=&3(J_3^1 + J_3^2 + J_3^3).
\end{align*}
Now, we derive the estimates for $J_3^1$, $J_3^2$ and $J_3^3$ separately. For the $J_3^1$,
	\begin{align*}
		J_3^1 &\le \frac{1}{N^2}\int_0^t \| \nabla(\nabla W_\varepsilon) \|_{L^\infty(\R^d)}^2\mathbb{E}\Big[\Big(\sum_{j=1}^{N}|X_{\xi,N}^j(s)-\overline{X}_\xi^j(s)|\Big)^2\Big] ds\\
		&\le \frac{1}{N\varepsilon^{d+2}}\int_0^t \mathbb{E}\big[ \sum_{j=1}^{N}\big|X_{\xi,N}^j(s)-\overline{X}_\xi^j(s)\big|^2\big] ds
		\le \frac{1}{\varepsilon^{d+2}}\int_0^t \mathbb{E}[\mathcal{M}^i(s)]ds,
	\end{align*}
where $\|\nabla(\nabla W_\varepsilon) \|_{L^\infty(\R^d)}\le\varepsilon^{-(d+2)/2}$. The $J_3^2$ can be handled
	\begin{align*}
		J_3^2 &\le \frac{1 }{N^2}\int_0^t\| \nabla(\nabla W_\varepsilon) \|_{L^\infty(\R^d)}^2\mathbb{E}\Big[\Big(\sum_{j=1}^{N}\big|X_{\xi,N}^i(s)-\overline{X}_{\xi}^i(s)\big|\Big)^2\Big] ds\\
		&\le \frac{1}{N^2\varepsilon^{d+2}}\int_0^t \mathbb{E}\big[N^2 |X_{\xi,N}^i(s)-\overline{X}_{\xi}^i(s)|^2\big] ds
		\le  \frac{1}{\varepsilon^{d+2}} \int_0^t \mathbb{E}[\mathcal{M}^i(s)]ds.
	\end{align*}
The $J_3^3$ is estimated as follows
	\begin{align*}
		J_3^3
		=& \frac{1}{N^2}\int_{0}^t
		\mathbb{E}\Big[\Big|\sum_{j=1}^N \Big(\nabla W_\varepsilon(\overline{X}_{\xi}^i(s)-\overline{X}_\xi^j(s)) - \nabla W_\varepsilon*\rho(\overline{X}_{\xi}^i(s),s)\Big)\Big|^2\Big]ds\\
		\le\ & \frac{1}{N^2}\int_0^T \mathbb{E}\Big[\sum_{j=1}^N\Big(\nabla W_\varepsilon(\overline{X}_{\xi}^i(s)-\overline{X}_\xi^j(s)) - \nabla W_\varepsilon*\rho(\overline{X}_{\xi}^i(s)) \Big)\\
		&\qquad\qquad\qquad\qquad\sum_{l=1}^{N}\Big( \nabla W_\varepsilon(\overline{X}_{\xi}^i(s)-\overline{X}_\xi^l(s))-\nabla W_\varepsilon*\rho(\overline{X}_{\xi}^i(s))\Big)\Big] dt\\
		=&\ \frac{1}{N^2}\sum_{j=1}^{N}\sum_{l=1}^{N} \int_0^t \mathbb{E}\Big[\Big(\nabla W_\varepsilon(\overline{X}_{\xi}^i(s)-\overline{X}_\xi^j(s)) - \nabla W_\varepsilon*\rho(\overline{X}_{\xi}^i(s))\Big) \\
		&\qquad\qquad\qquad\qquad\Big(\nabla W_\varepsilon(\overline{X}_{\xi}^i(s)-\overline{X}_\xi^l(s)) - \nabla W_\varepsilon*\rho(\overline{X}_{\xi}^i(s)) \Big)\Big] ds,
	\end{align*}
where for $j\ne l$ the expectation is zero. Hence,
\begin{align*}
|J_3^3| \le& \frac{1 }{N^2}\sum_{j=1}^{N} \int_0^t \mathbb{E}\Big[\Big|\nabla W_\varepsilon(\overline{X}_{\xi}^i(s)-\overline{X}_\xi^j(s)) - \nabla W_\varepsilon*\rho(\overline{X}_{\xi}^i(s),s)\Big|^2\Big]ds 
\le \frac{t\parallel f\parallel_{L^1(\R^{2d})}^2}{N\varepsilon^{d}},
\end{align*}
while exploiting the fact that 
$$\|\nabla W_\varepsilon*\rho\|_{L^\infty(\R^d)} \le \|\nabla W_\varepsilon\|_{L^\infty(\R^d)}\|\rho\parallel_{L^1(\R^d)} \le\frac{1}{\varepsilon^{d/2}} \| f\|_{L^1(\R^{2d})}.$$ 
Combining $J_3^1$, $J_3^2$ and $J_3^3$, we have
\begin{align}\label{II}
J_3 \le  \frac{2}{\varepsilon^{d+2}} \int_0^t \mathbb{E}[\mathcal{M}^i(s)]ds+ \frac{t\parallel f\parallel_{L^1(\R^{2d})}^2}{N\varepsilon^{d}}.
\end{align}

Thus, bring $J_1$, $J_2$ and $J_3$ into the inequality (\ref{N}), we can obtain
\begin{align*}
\sup \limits_{i=1\dots N}\mathbb{E}[\mathcal{N}^i(t)] &\leq   C\Big(t\big(\frac{1}{\delta^2\nu^4\varepsilon^{4d+2}} +\frac{1}{\varepsilon^{d+2}}+1\big) \int_0^t\sup \limits_{i=1\dots N}\mathbb{E}\big[\mathcal{M}^i(s)\big] ds\\
&\qquad+ t\big(\frac{1}{\nu^4\varepsilon^{4d}}+1\big)
\int_0^t\sup \limits_{i=1\dots N}\mathbb{E}\big[\mathcal{N}^i(s)\big] ds
+ t^2\big(\frac{1}{N\delta^2\nu^4\varepsilon^{4d}} + \frac{1}{N\varepsilon^{d}}\big)\Big)\\
&\leq   C\Big( \frac{t^2}{\delta^2\nu^4\varepsilon^{4d+2}} 
\int_0^t\sup \limits_{i=1\dots N}\mathbb{E}\big[\mathcal{N}^i(s)\big] ds+\frac{t^2}{N\delta^2\nu^4\varepsilon^{4d}} \Big),
\end{align*}
where $C$ is a constant depending on $\gamma$, $\lambda$, $\beta$ and $\| f_0\|_ {L^1(\R^{2d})}$. By the $\mathrm{Gr\ddot{o}nwall}$'s inequation and taking the supremum in time on both sides, we get
    \begin{align*}
 \sup\limits_{0<t<T}\sup \limits_{i=1\dots N}\mathbb{E}[\mathcal{N}^i(t)] &\leq \frac{CT^2}{N\delta^2\nu^4\varepsilon^{4d}}\Big(1+\frac{T^3}{\delta^2\nu^4\varepsilon^{4d+2}}
\exp\frac{T^3}{\delta^2\nu^4\varepsilon^{4d+2}}
\Big).
    \end{align*}
Further, we have the following estimate
    \begin{align*}
  \sup\limits_{0<t<T}\sup \limits_{i=1\dots N}  \mathbb{E}[\mathcal{M}^i(t)] &\le T\int_0^T\mathbb{E}[\mathcal{N}^i(s)]ds
		\leq  \frac{CT^3}{N\delta^2\nu^4\varepsilon^{4d}}\Big(1+\frac{T^3}{\delta^2\nu^4\varepsilon^{4d+2}}
\exp\frac{T^3}{\delta^2\nu^4\varepsilon^{4d+2}}
\Big).
    \end{align*}
    
Finally, for any $0<\alpha \ll 1$, we can choose $\varepsilon$, $\nu$ and $\delta$ so small that $1/(\delta^2\nu^4\varepsilon^{4d+2}) \sim \ln(N^\alpha)$, having
\begin{align*}
    \sup\limits_{0<t<T}\sup \limits_{i=1\dots N}  \mathbb{E}[\mathcal{M}^i(t)] &\leq
\frac{C T^3\ln N^\alpha}{N}\big(1+ T^3 N^{T^3\alpha }\ln (N^{\alpha})\big), \\
 \sup\limits_{0<t<T}\sup \limits_{i=1\dots N} \mathbb{E}[\mathcal{N}^i(t)] &\leq
\frac{CT^2\ln N^\alpha}{N}\big(1+ T^3 N^{T^3\alpha }\ln (N^{\alpha})\big),
    \end{align*}
where $C$ is a constant depending on $\gamma$, $\lambda$, $\beta$ and $\| f_0\|_ {L^1(\R^{2d})}$.
\end{proof}

\bibliographystyle{plain}
\bibliography{mflref.bib}

\end{document}